\documentclass{article}
\usepackage{amsmath,amssymb}
\usepackage{hyperref}
\usepackage{amsfonts}
\usepackage{color,graphicx}
\usepackage{comment}
\newtheorem{theorem}{Theorem}

\newtheorem{assumption}[theorem]{Assumption}

\newtheorem{corollary}[theorem]{Corollary}

\newtheorem{definition}[theorem]{Definition}

\newtheorem{lemma}[theorem]{Lemma}

\newtheorem{proposition}[theorem]{Proposition}
\newtheorem{remark}[theorem]{Remark}

\newenvironment{proof}[1][Proof]{\noindent\textbf{#1.} }{\ \rule{0.5em}{0.5em}}
\begin{document}

\title{{ Ulam Approximation for Nonautonomous Systems: Equivariant Measures and Linear Response}
}

\author{    Stefano Galatolo$^{1}$\footnote{Email: \texttt{stefano.galatolo@unipi.it}}
    Valerio Lucarini$^{2,3}$\footnote{Email: \texttt{v.lucarini@leicester.ac.uk}}, Isaia Nisoli$^{4}$\footnote{Email: \texttt{nisoli@im.ufrj.br}}\\[2ex]
    $^{1}$Dipartimento di Matematica, Universit\`a di Pisa, Pisa, Italy\\
    $^{2}$School of Computing and Mathematical Sciences\\ University of Leicester, Leicester, UK\\
    $^{3}$School of Sciences, Great Bay University\\Dongguan, P.R. China \\
    $^{4}$Universidade Federal do Rio de Janeiro, Rio de Janeiro, Brasil 
}

\maketitle

\begin{abstract}
{ Despite the prevalence of nonautonomous systems in applications, their statistical properties are much less understood than in the autonomous setting. Building on recent results on response theory for nonautonomous systems, we study the approximation of equivariant families and of their linear response by Ulam-type finite-dimensional reductions. First, we show that coarse-graining procedures associated with the classical Ulam method, and more generally with suitable finite-element projections, {such as those often associated with the Koopman/Kolmogorov operator-informed analysis of dynamical systems}, provide rigorous approximation of equivariant families for sequential systems with memory loss. 
Second, for systems whose transfer operators are regularizing, we prove that the linear response of the reduced finite-state Markov model converges to the projected linear response of the original system. 
We {\color{black} illustrate the results } with numerical experiments on simple but representative time-dependent diffusive models. These results provide a rigorous foundation for the use of Markov approximations in the study of statistical properties of nonautonomous  complex systems which  almost invariably relies on finite-scale and finite-precision descriptions of their states and dynamics.}
\end{abstract}


\section{Introduction}


Many systems of interest in the natural sciences, engineering, and complex systems theory are inherently {\em nonautonomous}: their effective dynamics changes in time because of external forcings, slow parameter drift, unresolved scales, or explicit control protocols.
{\color{black}Because of their intrinsic interest and relevance to real-world applications, nonautonomous dynamical systems have attracted increasing mathematical attention in recent years}   \cite{Ashwin2026,Crisan2026}.

{\color{black} For the study of the statistical properties of} such systems, the classical stationary viewpoint based on a single invariant measure is in general inadequate. As emphasized in \cite{Lucarini2026,GL26}, there is typically no time-independent reference distribution describing the system’s statistical state. Under suitable assumptions, the appropriate {\color{black}statistical} object is instead a time-indexed family of measures {\color{black}comprising of} the snapshot statistics of the system at each instant, often formalized as a pullback, covariant, or equivariant family {\color{black}(see Definition \ref{def1})} \cite{Bodai2013,Tel2020,Crauel1997,kloeden1997nonautonomous,kloeden2011nonautonomous,Chekroun2011,g253103}. Operationally, pullback-well posedness means that if one initializes an ensemble sufficiently far in the past and evolves it under the same time-dependent protocol, then the ensemble statistics observed at a given time eventually become essentially independent of the initial distribution{\color{black}, whilst instead converging to a time-dependent object (see Remark \ref{R4}). The time-indexed collection of measures obtained in such a way} provides a natural nonautonomous analogue of the reference statistical state {\color{black}defined by the invariant measure in the stationary case and} form the basis for perturbative and response-theoretic questions in time-dependent settings \cite{Lucarini2026,GL26}. {\color{black}
It might be useful for the reader who is less familiar with the notion of equivariant family to consider the illustrative example presented in Appendix \ref{AppGauss}. }

{Nonautonomous dynamics is central in the modelling of temporal networks \cite{Holmes2012,Vespignani2012,Ginestra2023}, including social and spatial networks \cite{Ginestra2010},  \cite{Ginestra2024}, as well as in time-dependent consensus dynamics \cite{Leonie2021}. In ecology, nonautonomous dynamics captures seasonal forcing, climate change, and resource pulses that drive population shifts beyond equilibrium assumptions \cite{Summers2000,Jansen2000,Basak2024}. In finance and economics, it accounts for evolving volatility, policy changes, and behavioral heterogeneity, {\color{black}which are key to making more accurate} predictions of market regime shifts \cite{Toscani2006,Chakrabarti2021,Kohlrausch2024}. Neuronal dynamics is impacted by time-dependent stimuli, synaptic plasticity, and neuromodulation, which are critical for understanding spike-timing adaptations, learning, and pathological oscillations such as epileptic seizures \cite{Horrocks2024,Bolelli2025}. }

The climate system offers a particularly compelling motivation for studying time-dependent dynamical systems \cite{Crisan2026}. Since climate variability spans a vast range of time scales \cite{saltzman_dynamical-2,vonderHeydt2021,GhilLucarini2020,LucariniChekroun2023}, it can be difficult - even operationally - to define a rigorous steady reference state; nonetheless, standard detection-and-attribution methodologies often build upon a (hypothetical) steady baseline, such as ``preindustrial'' conditions \cite{Allen1999,Hegerl2011,Hannart2014,LucariniChekroun2024}. 

A time-dependent baseline is arguably more physically faithful: natural forcings (astronomical, volcanic, seasonal, etc.) can be incorporated into the reference dynamics, while anthropogenic contributions can be treated as additional perturbations relative to this nonstationary background. This shift in viewpoint is also consequential for attributing observed changes to acting forcings, because the objects mediating the effect of forcings (e.g. Green's functions-type response operators) inherit explicit time dependence, even though causality remains meaningful \cite{Lucarini2026} {\color{black}(see the discussion on the response of the system to external time dependent forcing at the beginning of Section \ref{S4}). }

Across all these domains, explicit time dependence is essential for capturing transient responses, resilience, and information processing in non-stationary environments. 
Thus, nonautonomous systems provide a natural mathematical framework for studying adaptation, regulation, response, predictability, and transient statistical behaviour in many real-world complex systems. {\color{black}It is also important to remark that considering} time-dependent systems allow for a broader class of critical behaviour that goes beyond the classical notions of tipping points associated with bifurcations  \cite{Ashwin2012}.

These considerations motivate the development of quantitative tools for nonautonomous statistical states: existence and uniqueness of the {\color{black}equivariant} family, and, crucially for applications and numerics, stability of the resulting equivariant/pullback measures under perturbations and approximations of the underlying time-dependent dynamics.

{\color{black}\subsection{Markov Modelling, finite element reduction}}

{\color{black}
In this paper, we build on our recent results \cite{Lucarini2026,GL26} and investigate the approximation of equivariant families and their response to perturbations using Ulam-type finite-dimensional schemes.

This problem is of considerable practical relevance to the analysis of complex systems, for which one often has to rely on finite-resolution and finite-precision representations of states and dynamics. 
}
{\color{black}The convergence and robustness of Ulam-type methods (see Section \ref{subsec:abstract_projection_estimates} for the characterization of the class of finite element discretizations we study in the paper) provides a robust basis for one of the key modelling strategies that is currently used for investigating complex and multiscale systems. Markov state modelling (MSM) \cite{Pande2010,Schuette2012,Husic2019} relies on discretising time and on partitioning the phase space of the systems of interest in a finite number of states, which are usually associated with the elements of a Voronoi tessellation \cite{Aurenhammer91} using k-means clustering \cite{Ikotun2023}. The coarse-grained dynamics of the system is then represented as a finite-state Markov chain \cite{Nor97} that describes the probability of the system to perform a transition between all pairs of states in the chosen time interval. 

MSM provides an equation-agnostic description of possibly high-dimensional systems and is supported by widely available scientific software and linear-algebra tools \cite{MATLAB2024,PYTHON,JULIA}.
Reducing a complex system to a Markov chain makes it comparatively easy to study spectral properties and response at a coarse-grained level, both in the stationary case \cite{Lucarini2016,Santos2020,Lucarini2025} and in the time-dependent one \cite{Lucarini2026}.

MSM has been successfully applied in many research areas. In biophysics, which is probably the area where MSM has been most successful, it is used to study protein folding \cite{Pande2010,Konovalov2021MSM}, conformational transitions \cite{sengupta2018markov}, ligand binding kinetics \cite{ge2021markov,ge2022estimation}, among others.  Beyond biophysics, MSM finds use in chemistry (reaction networks) \cite{anderson2011continuous,macdonald2023continuous}, climate science (regime shifts) \cite{Springeretal2024,mukhin2025metastability}, finance (market state transitions) \cite{zakamulin2023bull}, neuroscience (brain state dynamics) \cite{ou2015characterizing,ezaki2021modelling}, and web science (PageRank) \cite{page1999pagerank}.

Our work provides a clear theoretical support for this modelling practise, by establishing a set of precise assumptions under which the numerical or data-driven work done when using a Markov model is proved to yield reliable results.} 


\subsection{Scope of this Paper}

In this work we study nonautonomous deterministic and random dynamical systems through their associated transfer operators.
 Our main objects are \emph{equivariant families} of measures, namely time-dependent statistical states transported by a sequential composition of transfer operators. We investigate the extent to which such families can be approximated by finite-dimensional schemes of Ulam type, and more generally by suitable finite-element or finite-state Markov approximations.

A second goal is to study the approximation of \emph{linear response}  in the nonautonomous setting. More precisely, for random dynamical systems with suitably distributed nonsingular noise, we investigate whether the derivative of the equivariant family with respect to perturbations can itself be approximated by the linear response of a finite-dimensional Markov model.
Our aim is to prove results of the following type: "the linear response of the original system is approximated, in a suitable weak norm, by the linear response of a sufficiently fine Ulam-type discretization."
{\color{black}We are not aware of previous results of this kind, even in the more extensively studied autonomous setting.}
Our approach is based on the transfer-operator framework for nonautonomous systems developed in \cite{GL26}. In that framework, a sequential or random time-dependent system is encoded by a bi-infinite sequence of transfer operators, and the corresponding equivariant family is characterized as a fixed point of a \emph{global} transfer operator acting on a Banach space of sequences of measures. This point of view allows one to reformulate the linear response problem as a resolvent equation on sequence space, leading to explicit formulas and to a natural approximation strategy. See in particular \cite[Theorem~11]{GL26} for the abstract linear response result and \cite[Remark~12]{GL26} for the corresponding Neumann-series interpretation.

A central feature of the theory {\color{black} we present} is that it is largely axiomatic. The main approximation results are formulated under a set of assumptions on the underlying weak and strong spaces of measures, on the regularizing properties of the transfer operators, and on the consistency of the finite-dimensional projections. This level of abstraction makes it possible to treat deterministic and random systems within a common framework and to separate the functional-analytic mechanism from the verification of the assumptions in specific examples.
After developing the abstract theory, we show how it applies to concrete classes of systems. In particular, we consider nonautonomous deterministic systems and random systems with nonsingular noise for which one can establish both exponential memory loss and suitable smoothing properties. For such systems, the abstract assumptions can be verified, and one obtains the approximation of both the equivariant family and its linear response by finite-dimensional Markov models of Ulam type.
Among the {\color{black}main} assumptions, a key role is played by \emph{memory loss}, namely the fast forgetting of initial conditions along the nonautonomous evolution.
 In the present context, memory loss plays a role analogous to that of the spectral gap in the autonomous theory: it provides the contraction mechanism that underlies existence, stability, and linear response of equivariant families. 
 
 The remaining assumptions required for the approximation scheme, concerning in particular the regularizing action of the transfer operators and the compatibility of the projections used in the finite-dimensional approximation, are close in spirit to the ones appearing in the autonomous setting.
The study of memory loss for sequential systems has substantial literature. The viewpoint was prominently developed by ~\cite{OSY2009}, who proved exponential memory loss for important classes of time-dependent expanding systems. Related results for higher-dimensional piecewise expanding maps were obtained by ~\cite{GOT2013}, while ~\cite{Cui2021} established exponential memory loss and related statistical properties for sequential Lasota--Yorke-type interval maps.
A complementary line of work is based on transfer-operator and functional-analytic methods. Conze and Raugi~\cite{ConzeRaugi2007} developed a systematic operator approach to sequential expanding dynamics and derived several limit theorems. Building on this framework, \cite{AiminoRousseau2016} proved concentration inequalities for sequential interval maps under suitable memory-loss assumptions, ~\cite{HNTV2017} established strong stochastic approximations, including almost sure invariance principles, in non-stationary dynamical settings.
{\color{black}}
These works use different techniques, ranging from coupling and distortion estimates to spectral and transfer-operator methods, but they share a common principle: under suitable uniform assumptions, sequential systems display a strong form of statistical stability in time. This is also the mechanism that will be exploited here.

The approximation of dynamical systems transfer operators by finite-dimensional Markov models has a long history, beginning with Ulam's original idea \cite{Ulam1960} and Li's classical proof of convergence for expanding interval maps~\cite{Li1976}. In the autonomous setting, Ulam-type discretizations and related projection methods have been extensively developed for the approximation of invariant densities, spectral data, and metastable structures; see for instance \cite{BoseMurray2001}, \cite{Froyland2007} and \cite{GalatoloNisoli2014}, \cite{nisolinuovo}. Related finite-dimensional approximation strategies based on smoother bases, interpolation or spectral truncation, rather than on Ulam partitions, have also been developed; see for instance ~\cite{Wormell2019,GalatoloLopezVereauMarangioNisoli2025,BandtlowSlipantschuk2020,BandtlowJustSlipantschuk2016}.

In random settings, analogous questions have also been studied: \cite{Froyland1999} proved convergence of Ulam's method for a class of random interval maps, while \cite{FGQ2014} established stability and approximation results for random invariant densities of Lasota-Yorke map cocycles {\color{black}in the quenched setting,} including perturbations arising from finite-rank approximations of Ulam type. 
A related computer-assisted application of the finite-rank operator approximation in random dynamics was developed in \cite{Galatolo_2020}, where a certified \(L^1\)-approximation of the stationary measure, obtained through a Ulam-type discretization, is used to rigorously prove the existence of noise-induced order.

The present work fits into this line of research, but focuses on the genuinely nonautonomous viewpoint of equivariant families and on the approximation not only of the statistical state itself, but also of its linear response under perturbation.
In the autonomous setting, the approximation of linear response has also been studied from several complementary viewpoints. \cite{PollicottVytnova2016} derived explicit convergent series formulas for linear response and higher derivatives in expanding and hyperbolic settings. Closer in spirit to the present work, ~\cite{BahsounGalatoloNisoliNiu2018} developed a rigorous computational framework for linear response based on transfer operators and finite-rank approximation schemes, obtaining validated numerical approximations in concrete examples. 
A different direction is represented by Monte-Carlo and trajectory-based methods, such as the fast linear response and related algorithms ~\cite{Ni2020,Ni2025}, which aim at computing response information from orbit sampling rather than from finite-dimensional Markov reductions.

While these works contain important ideas related to operator approximation, numerical response, and stochastic perturbations, they do not provide a simple abstract statement of the form: "the linear response of the finite-state Markov reduction converges to the linear response of the original system".
One of the goals of the present work is to provide such a result in an abstract form that is directly applicable once the required assumptions are verified.

{\color{black} Finally, we complement the mathematical analysis with numerical simulations of simple but instructive models, illustrating the approximation of equivariant measures and their linear response by the Ulam method in the nonautonomous setting.}

{We remark that our results are not restricted to classical Ulam partitions with piecewise constant basis functions, but also cover more general finite-element approximation spaces. Hence, our work has direct relevance for the literature dedicated to the study of a large variety of deterministic and stochastic dynamical systems via \textit{Koopmanism} \cite{Budisic2012}, \textit{i.e.} by constructing  finite dimensional approximations of  the Koopman/Kolmogorov operator \cite{KutzBrunton2016,Klus2020,Colbrook2023_MultiverseDMD,Brunton2022,Klus2024}. {\color{black} This} research line bridges the gap between theory-informed and data-driven analysis of a large variety of systems \cite{Mauroy2020}. Specifically, our findings on the convergence of the approximation of the response operator provide support to recently published formal calculations showing that Green's functions can be expressed as a sum of terms each associated with individual eigenmodes of the Koopman/Kolmogorov operator of the unperturbed system \cite{Santos2022,LucariniChekroun2023,Lucarini2025,zagli_SIAM:2026,lucarini2025generalframeworklinkingfree}.}

\subsection{Structure of the Paper and Main Results}
We describe the structure of the paper and  the construction made, highlighting the main results.

In Section \ref{S2}  we introduce the basic transfer-operator framework for nonautonomous systems, together with the notions of equivariant family, global map on sequence spaces, and memory loss.

In Section \ref{S3} we show how a Markov model of the original system extracted by Ulam-like methods can approximate the equivariant measure of the original system. We also present quantitative results on how the approximation improves as we consider finer partitions of the phase space. 

We begin by showing a general quantitative stability result for the equivariant families. In Theorem \ref{thm:dlogd} we show that under the main assumption that the system has exponential memory loss, the 
system is, in a suitable sense, statistically stable under small perturbations, and the response to a perturbation of size $\delta$ is $O(\delta \ log\delta)$.

In Theorem \ref{thm:lipschitz} we also prove  a more stringent Lipschitz rate that applies in the case of regularizing transfer operators.
This case is common for system whose dynamics include diffusive terms, as in the case of stochastic differential equations, or random system with (state dependent) additive noise.
This situation is common in many applications across complex systems science, 
where models are often expressed as random dynamical systems.

In Section \ref{subsec:abstract_projection_estimates} we apply these general results to the Ulam-like approximations.
In Proposition \ref{prop:projection_delta_log_delta} we  show that the
rate of approximation is of the order of $\delta log \delta$.
This is the same rate expected in the autonomous case (see \cite{BoseMurray2001}).

In Section \ref{subsec:ulam_loss_of_memory} we show that a sufficiently accurate Ulam discretization must inherit exponential memory loss.

In Section \ref{S4} we suppose that {\color{black} a family of transfer operators varies smoothly} {\color{black} and the operators} are regularizing (see {\color{black} Assumptions {\bf (H1)},...,{\bf (H3)})}; we consider the linear response of the system to small perturbations and we show how this can be approximated by the linear response of a Markov model given by a suitable Ulam-like approximation (Proposition \ref{prop:linear_response_approx}).

In Sections \ref{subsec:uniformly_positive_smoothing} 
and \ref{subsec:random_systems_satisfying_assumptions}
we present a rather general family of  sequential random systems with regularizing noise for which the abstract framework developed in the paper applies. 

In Section \ref{numericalsim} we illustrate the theoretical findings on concrete examples of spatial and temporal time discretizations of diffusive processes undergoing a time dependent forcing.

In Section \ref{sec:conclu} we present our conclusions and perspective for future work.

{\color{black} In Appendix \ref{AppGauss} we present a simple example  of equivariant family in order to illustrate the concept.}

In Appendix \ref{subsec:ulam_bv_l1} we collect well known estimates, showing that the classical Ulam approximation scheme satisfies the properties for a Markov model induced by a finite element projection we assume in the paper.

\section{Sequential non autonomous dynamical systems and equivariant measures}\label{S2}

In this section, we recall the basic notions about nonstationary sequences of transfer operators, equivariant measures, and memory loss (see \cite{GL26} Sections 2-3 for more details).

Let $(X,\mathcal A)$ be a measurable space. We consider two normed linear spaces of finite signed measures
\[
(B_s,\|\cdot\|_s)\subset (B_w,\|\cdot\|_w),
\qquad\text{with}\qquad
\|\mu\|_w\le \|\mu\|_s \ \ \forall \mu\in B_s
\]
and we assume the inclusion $B_s\hookrightarrow B_w$ is continuous.

The space \(B_s\) plays the role of a strong space, carrying enough regularity to control approximation and perturbation arguments, while \(B_w\) is the weak space in which stability estimates are formulated.

Denote by $\mu(X)$ the total mass of $\mu$ and define the zero-mass subspaces
\[
V_s:=\{\mu\in B_s:\mu(X)=0\},
\qquad
V_w:=\{\mu\in B_w:\mu(X)=0\}.
\]
We also set
\[
\mathcal P_s:=\{\mu\in B_s:\ \mu\ge 0,\ \mu(X)=1\}.
\]

\paragraph{Nonautonomous Markov operators}
We will now study the dynamics of a certain random or deterministic system through the properties of its transfer operators.
A (two-sided) nonautonomous system of transfer operators is a bi-infinite family
\[
(L_n)_{n\in\mathbb Z},\qquad L_n:B_w\to B_w,\ \ L_n:B_s\to B_s,
\]
such that each $L_n$ is \emph{Markov} (preserving mass and preserving positivity). We will use the standard forward compositions
\[
L^{(j,j+k-1)}:=L_{j+k-1}\circ \cdots \circ L_{j+1}\circ L_j,
\qquad k\ge 1,
\]
and $L^{(j,j-1)}:=\mathrm{Id}$.

\paragraph{Sequence space and the global map}
Define the sequence spaces
\[
\mathcal B_s:=\ell^\infty(\mathbb Z;B_s),\qquad
\|\boldsymbol\mu\|_{\mathcal B_s}:=\sup_{n\in\mathbb Z}\|\mu_n\|_s,
\]
and similarly $\mathcal B_w:=\ell^\infty(\mathbb Z;B_w)$.

Given $(L_n)_{n\in\mathbb Z}$, define the \emph{global map} (an autonomous map on sequences)
\[
\mathbb F_L:\mathcal B_s\to\mathcal B_s,
\qquad
(\mathbb F_L\boldsymbol\mu)_n:=L_{n-1}\mu_{n-1},
\qquad n\in\mathbb Z.
\]

\begin{definition}[Equivariant measures / equivariant family]\label{def1}
A sequence $\boldsymbol\mu=(\mu_n)_{n\in\mathbb Z}\in \mathcal P_s^{\mathbb Z}$ such that $\sup_n\|\mu_n\|_s<\infty$ is said to be an \emph{equivariant measure} (or \emph{equivariant family}) for $(L_n)$ if
\begin{equation}\label{eq:equivariance}
\mu_{n+1}=L_n\mu_n\qquad\forall n\in\mathbb Z.
\end{equation}
\end{definition}

Directly from the definition of $\mathbb F_L$ we have the following

\begin{lemma}[Fixed points vs.\ equivariant families]\label{lem:fixedpoint_equivariant}
For $\boldsymbol\mu\in\mathcal B_s$ we have
\[
\mathbb F_L(\boldsymbol\mu)=\boldsymbol\mu
\quad\Longleftrightarrow\quad
\boldsymbol\mu \text{ satisfies \eqref{eq:equivariance}}.
\]
Moreover, $\mathbb F_L(\mathcal P_s^{\mathbb Z})\subset \mathcal P_s^{\mathbb Z}$.
\end{lemma}


\paragraph{Memory loss}\label{S2.3}

The following definition captures the idea that the nonautonomous evolution ``forgets'' initial conditions.

\begin{definition}[Strong Memory loss]\label{def:lom}
We say that $(L_n)_{n\in\mathbb Z}$ has strong exponential \emph{memory loss  on $V_s$} if there exists  constants $C\ge 1$ and $\lambda>0$ such that for all $j\in\mathbb Z$, all $k\ge 1$, and all $v\in V_s$,
\begin{equation}\label{eq:lom_general}
\bigl\|L^{(j,j+k-1)} v\bigr\|_s \;\le\; C e^{-\lambda k} \,\|v\|_s .
\end{equation}
\end{definition}

\begin{remark}\label{R4}
Since differences of probability measures have zero mass, \eqref{eq:lom_general} implies
\[
\|L^{(j,j+k-1)}\nu - L^{(j,j+k-1)}\nu'\|_s \le C e^{-\lambda k}\|\nu-\nu'\|_s
\quad\forall \nu,\nu'\in\mathcal P_s.
\]
Thus  memory loss is a uniform contraction property on the ``zero-mass directions''.
\end{remark}
From memory loss, it is easy (see  \cite{GL26} Section 3 for more details) to obtain existence and uniqueness results of the equivariant family for the system.
In the following sections, this contraction mechanism will be the basic ingredient both for quantitative stability under perturbations and for the approximation of equivariant families by finite-dimensional reductions.

\section{Quantitative stability under perturbations}\label{S3}

We now consider two  nearby non autonomous systems $(L_n)$ and $(\widetilde L_n)$ having exponential memory loss and estimate the distance between the  equivariant families of the two systems. 
By this we quantitatively understand the stability of the equivariant measure as the system is perturbed.

\subsection{Standing perturbation hypotheses}
For the two systems we assume:
\begin{assumption}[Weak contraction of the weak norm]\label{ass:weak_nonexp}
For all $n\in\mathbb Z$ and all $\mu\in B_w$,
\[
\|L_n\mu\|_w\le \|\mu\|_w,
\qquad
\|\widetilde L_n\mu\|_w\le \|\mu\|_w.
\]
\end{assumption}

\begin{assumption}[Exponential memory loss]\label{ass:elom}
 $(L_n)$ satisfy strong exponential memory loss (Definition \ref{def:lom}).
\end{assumption}

\begin{assumption}[Mixed norm closeness]\label{ass:mixed_delta}
Define
\[
\delta:=\sup_{n\in\mathbb Z}\|L_n-\widetilde L_n\|_{B_s\to B_w}={\color{black} \sup_{f\in B_s\setminus\{0\}}
\frac{\|[L_n-\widetilde L_n]f\|_w}{\|f\|_s}}
<\infty.
\]
\end{assumption}

\begin{assumption}[Uniform strong bounds of the equivariant families]\label{ass:Ms}
Let $\boldsymbol\mu,\widetilde{\boldsymbol\mu}\in\mathcal P_s^{\mathbb Z}$ be bounded fixed points of
$\mathbb F_L$ and $\mathbb F_{\widetilde L}$, respectively, and assume
\[
M_s:=\sup_{n\in\mathbb Z}\|\mu_n\|_s<\infty,
\qquad
\widetilde M_s:=\sup_{n\in\mathbb Z}\|\widetilde\mu_n\|_s<\infty.
\]
(These bounds are typically obtained from a uniform Lasota--Yorke inequality, {\color{black} see next Assumptions {\bf (LY),(W)}}.)
\end{assumption}

\subsection{Exponential memory loss and quantitative statistical stability }

The next result shows that if two nonautonomous systems are close in the mixed strong-to-weak norm, and the reference system has exponential memory loss, then their equivariant families are close in the weak norm. {\color{black} This will be used in next section, considering a Ulam-like discretization as a suitable small perturbation of the system.
In particular, the result shows that} the equivariant family continuously depends on the operators sequence, and one obtains an explicit modulus of continuity {\color{black} for the dependence of the equivariant measures on the operators} the of order \(\delta|\log\delta |\)  {\color{black} when the size of the perturbation is of the order of $\delta$}. 
{\color{black} We remark that this modulus of continuity is known to be optimal for autonomous
dynamical systems having exponential convergence to equilibrium, which corresponds to our exponential strong memory loss in our case (see \cite{Baladi2007,Mazzolena2006,JEP_2018__5__377_0}). 
In the closely related autonomous context of quenched random dynamical systems a similar result appears in \cite[Theorem 8]{DRAGIČEVIĆ_SEDRO_2023}. }

\begin{theorem}[$\delta\log\delta$ stability in the weak norm]\label{thm:dlogd}
Assume Assumptions~\ref{ass:weak_nonexp}--\ref{ass:Ms}.
Then there exists a constant $K>0$ (depending only on $C,\lambda,M_s,\widetilde M_s$)
such that, for all $0<\delta<1$,
\begin{equation}\label{eq:dlogd_bound}
\sup_{n\in\mathbb Z}\|\mu_n-\widetilde\mu_n\|_w
\ \le\ 
K\,\delta\bigl(1+\log\tfrac{1}{\delta}\bigr).
\end{equation}
In particular, the fixed equivariant family depends continuously on the operator sequence in the mixed norm.
\end{theorem}

\begin{proof}
Fix $n\in\mathbb Z$ and set $\Delta_n:=\mu_n-\widetilde\mu_n$. Since both are probabilities, $\Delta_n(X)=0$.
Using equivariance $\mu_n=L_{n-1}\mu_{n-1}$ and $\widetilde\mu_n=\widetilde L_{n-1}\widetilde\mu_{n-1}$ we get
\begin{equation}\label{eq:recurrence}
\Delta_n
= L_{n-1}\Delta_{n-1} + E_{n-1},
\qquad
E_{n-1}:=(L_{n-1}-\widetilde L_{n-1})\widetilde\mu_{n-1}.
\end{equation}
By Assumption~\ref{ass:mixed_delta} and the bound $\|\widetilde\mu_{n-1}\|_s\le \widetilde M_s$,
\begin{equation}\label{eq:Ebound}
\|E_{n-1}\|_w
\le \|L_{n-1}-\widetilde L_{n-1}\|_{B_s\to B_w}\,\|\widetilde\mu_{n-1}\|_s
\le \delta\,\widetilde M_s.
\end{equation}

Iterating \eqref{eq:recurrence} backwards $N$ steps yields the Duhamel expansion
\begin{equation}\label{eq:duhamelN}
\Delta_n
=
L^{(n-N,n-1)}\Delta_{n-N}
+\sum_{j=0}^{N-1} L^{(n-j,n-1)} E_{n-j-1}.
\end{equation}

We now estimate the first summand of \eqref{eq:duhamelN}.
Since $\Delta_{n-N}\in V_s$ and $\|\Delta_{n-N}\|_s\le \|\mu_{n-N}\|_s+\|\widetilde\mu_{n-N}\|_s\le M_s+\widetilde M_s$,
Assumption~\ref{ass:elom} gives
\[
\|L^{(n-N,n-1)}\Delta_{n-N}\|_w
\le \|L^{(n-N,n-1)}\Delta_{n-N}\|_s
\le C e^{-\lambda N}(M_s+\widetilde M_s).
\]

We now estimate the second summand of \eqref{eq:duhamelN}.
By weak nonexpansion (Assumption~\ref{ass:weak_nonexp}), each composition is nonexpanding on $B_w$:
\[
\|L^{(n-j,n-1)}\eta\|_w\le \|\eta\|_w\quad\forall \eta\in B_w.
\]
Hence, using \eqref{eq:Ebound},
\[
\left\|\sum_{j=0}^{N-1} L^{(n-j,n-1)} E_{n-j-1}\right\|_w
\le \sum_{j=0}^{N-1}\|E_{n-j-1}\|_w
\le N\,\delta\,\widetilde M_s.
\]

Combining these bounds with \eqref{eq:duhamelN} we obtain, uniformly in $n$,
\begin{equation}\label{eq:two_terms}
\|\Delta_n\|_w \le C e^{-\lambda N}(M_s+\widetilde M_s) + N\,\delta\,\widetilde M_s.
\end{equation}
Choose
\[
N:=\left\lceil \frac{1}{\lambda}\log\frac{1}{\delta}\right\rceil
\qquad (0<\delta<1).
\]
Then $e^{-\lambda N}\le \delta$ and $N\le 1+\frac{1}{\lambda}\log\frac{1}{\delta}$. Substituting into \eqref{eq:two_terms}
yields \eqref{eq:dlogd_bound} with a suitable constant $K$.
\end{proof}

\subsection{Lipschitz stability under regularisation $B_w\to B_s$}
A more stringent Lipschitz stability estimate can be provided under the additional assumption that the transfer operators have regularizing behavior, from the weak space to the strong one.
\begin{assumption}[One-step regularisation]\label{ass:regularising}
There exists $R>0$ such that for all $n\in\mathbb Z$ and all $\eta\in B_w$,
\[
\|L_n\eta\|_s\le R\|\eta\|_w.
\]
\end{assumption}

\begin{remark}
Assumption~\ref{ass:regularising} is a strong form of ``smoothing'', corresponding to a Lasota--Yorke inequality with $\lambda_1=0$.
It is natural in many random/noisy examples where the transfer operator includes an explicit smoothing kernel.
\end{remark}

\begin{theorem}[Lipschitz stability in the weak norm under regularisation]\label{thm:lipschitz}
Under Assumptions~\ref{ass:weak_nonexp}, \ref{ass:elom}, \ref{ass:mixed_delta}, \ref{ass:Ms}, and \ref{ass:regularising}.
It follows that
\begin{equation}\label{eq:lipschitz_bound}
\sup_{n\in\mathbb Z}\|\mu_n-\widetilde\mu_n\|_w
\ \le\
\widetilde M_s\left(1+\frac{C R}{1-e^{-\lambda}}\right)\delta.
\end{equation}
{\color{black}Were $C,\lambda$ are the constants in Definition \ref{def:lom}.}
In particular, the dependence on the mixed perturbation size $\delta$ is \emph{Lipschitz}.
\end{theorem}

\begin{proof}
Fix $n\in\mathbb Z$ and use \eqref{eq:duhamelN}. As in the proof of Theorem~\ref{thm:dlogd},
the first term satisfies
\[
\|L^{(n-N,n-1)}\Delta_{n-N}\|_w
\le \|L^{(n-N,n-1)}\Delta_{n-N}\|_s
\le C e^{-\lambda N}(M_s+\widetilde M_s)\xrightarrow[N\to\infty]{}0.
\]
Thus letting $N\to\infty$ in \eqref{eq:duhamelN} yields the convergent series representation in $B_w$:
\begin{equation}\label{eq:series_rep}
\Delta_n=\sum_{j=0}^{\infty} L^{(n-j,n-1)}E_{n-j-1}.
\end{equation}

The $j=0$ satisfies $\|E_{n-1}\|_w=\|(L_{n-1}-\widetilde L_{n-1})\widetilde\mu_{n-1}\|_w\le \delta\,\widetilde M_s$ by \eqref{eq:Ebound}.
For $j\ge 1$, write
\[
L^{(n-j,n-1)}E_{n-j-1}
=
L^{(n-j+1,n-1)}\bigl(L_{n-j}E_{n-j-1}\bigr).
\]
Since $E_{n-j-1}(X)=0$ and $L_{n-j}$ preserves mass, $L_{n-j}E_{n-j-1}\in V_s$.
Using exponential memory loss and then regularisation
\[ \|L^{(n-j,n-1)}E_{n-j-1}||_s
\le C e^{-\lambda(j-1)}\|L_{n-j}E_{n-j-1}\|_s
\le C R e^{-\lambda(j-2)}\|E_{n-j-1}\|_w.
\]
Finally, \eqref{eq:Ebound} gives $\|E_{n-j-1}\|_w\le \delta\,\widetilde M_s$, hence for $j\ge 1$,
\[
\|L^{(n-j,n-1)}E_{n-j-1}\|_w
\le C R e^{-\lambda(j-2)}\,\delta\,\widetilde M_s.
\]
Summing \eqref{eq:series_rep} {\color{black}by, the classical geometric series estimates}
\[
\|\Delta_n\|_w
\le \delta\,\widetilde M_s + \sum_{j=1}^\infty C R e^{-\lambda(j-{\color{black} 2})}\,\delta\,\widetilde M_s
=
\widetilde M_s\left(1+\frac{C R}{1-e^{-\lambda}}\right)\delta.
\]
The bound is uniform in $n$, proving \eqref{eq:lipschitz_bound}.
\end{proof}


\section{Finite element reduction and approximation rate for the equivariant family}
\label{subsec:abstract_projection_estimates}

Here we introduce finite-dimensional approximation schemes in an abstract way. We view them as small perturbations of the original nonautonomous system. As in the classical Ulam framework, the basic construction is obtained by composing the transfer operators with a finite-rank projection {\color{black} satisfying certain axiomatic properties}. This point of view allows one to treat, within a common setup, the original Ulam discretization based on projections  to piecewise constant functions, and more general finite-element schemes involving smoother projections. In this sense, the framework considered below may be regarded as a natural extension of the classical Ulam approximation method to a broader class of projection-based discretizations.
 We start from basic axiomatic estimates for finite element discretizations, we then apply the resulting framework to the finite elements approximation of equivariant families, establishing quantitative estimates for the approximation rate.

The definitions and basic properties regarding the classical Ulam approximation scheme, {\color{black} where densities are approximated by piecewise constant ones,} are recalled in Appendix \ref{subsec:ulam_bv_l1}, {\color{black} where we also prove that this approximation scheme satisfy the axiomatic properties assumed in this Section}.
 
Below we list a set of rather usual assumptions  for the operators and the finite rank projections we consider to define a Ulam-like finite element reduction. 
Let \(L_n:B_w\to B_w\) and \(L_n:B_s\to B_s\) be a family of linear operators, acting on a strong and weak space as in the previous sections. Let \((\pi_m)_{m\ge1}\) be a family of projections on \(B_w\). We assume:

\medskip

\noindent
\textbf{(P1)} there exist constants \(C_w,C_s>0\) such that
\[
\|\pi_m f\|_w\le C_w\|f\|_w
\qquad \forall f\in B_w,
\]
and
\[
\|\pi_m f\|_s\le C_s\|f\|_s
\qquad \forall f\in B_s,
\]
for every \(m\ge1\);

\medskip

\noindent
\textbf{(P2)} there exists a sequence \(\delta_m\to0\) such that
\[
\|\pi_m f-f\|_w\le \delta_m\|f\|_s
\qquad \forall f\in B_s;
\]

\medskip

\noindent
\textbf{(LY)} there exist constants \(\lambda \in (0,1),B\ge1\) such that \footnote{\color{black} An estimate of this form is commonly referred to as a Lasota-Yorke inequality or a Doeblin-Fortet inequality. It expresses that the transfer operator contracts the component of the strong norm measuring regularity, up to a remainder controlled by the weak norm. This is a weak form of regularization of the transfer operator, which can be established also for suitable classes of deterministic transfer operators.
In the autonomous setting, together with the compact embedding of the strong space into the weak one, such an estimate is a classical ingredient in proving the quasi-compactness of the transfer operator on the strong space. In the nonautonomous setting, it provides control of the regularity along compositions of transfer operators.
} for each $n$
\[
\|L_nf\|_s\le \lambda \|f\|_s+B\|f\|_w
\qquad \forall f\in B_s;
\]

\medskip

\noindent
\textbf{(W)} there exists \(C_L>0\) such that for each $n$
\[
\|L_nf\|_w\le C_L\|f\|_w
\qquad \forall f\in B_w.
\]

For each \(m,n\), define the finite-rank approximation as
\[
M_n^{(m)}:=\pi_m L_n\pi_m,
\qquad n\in\mathbb Z.
\]

\begin{proposition}[Abstract Lasota--Yorke and mixed-distance estimates]
\label{cor:abstract_sequential_projection}
Under assumptions \textbf{(P1)}--\textbf{(W)},  for each $n,m$ the operators \(M_n^{(m)}\) satisfy:

\begin{enumerate}
\item[\textup{(i)}] a uniform Lasota--Yorke inequality on \(B_s\),
\[
\|M_n^{(m)} f\|_s
\le
C_s\lambda C_s\,\|f\|_s + C_sBC_w\,\|f\|_w
\qquad \forall f\in B_s;
\]

\item[\textup{(ii)}] a mixed strong-to-weak approximation estimate,
\[
\|M_n^{(m)} f-L_nf\|_w
\le
\delta_m\Bigl[(C_wC_L+\lambda)\|f\|_s+B\|f\|_w\Bigr]
\qquad \forall f\in B_s.
\]
In particular,
\[
\|M_n^{(m)}-L_n\|_{B_s\to B_w}
\longrightarrow 0
\qquad\text{as }m\to\infty.
\]
\end{enumerate}
\end{proposition}

\begin{proof}
For the Lasota--Yorke estimate, let \(f\in B_s\). Then
\[
\|M_n^{(m)} f\|_s
=
\|\pi_m L_n\pi_m f\|_s
\le
C_s\|L_n\pi_m f\|_s.
\]
Applying \textbf{(LY)} to \(\pi_m f\), we get
\[
\|L_n\pi_m f\|_s
\le
\lambda\|\pi_m f\|_s+B\|\pi_m f\|_w.
\]
Using \textbf{(P1)} again,
\[
\|M_n^{(m)} f\|_s
\le
C_s\lambda C_s\,\|f\|_s + C_sBC_w\,\|f\|_w.
\]

For the mixed estimate, write
\[
M_n^{(m)} f-L_nf
=
\pi_mL_n\pi_m f-L_nf
=
\pi_mL_n(\pi_m f-f)+(\pi_mL_nf-L_nf).
\]
Hence
\[
\|M_n^{(m)} f-L_nf\|_w
\le
\|\pi_mL_n(\pi_m f-f)\|_w+\|\pi_mL_nf-L_nf\|_w.
\]
For the first term, using \textbf{(P1)}, \textbf{(W)}, and then \textbf{(P2)},
\[
\|\pi_mL_n(\pi_m f-f)\|_w
\le
C_w\|L_n(\pi_m f-f)\|_w
\le
C_wC_L\|\pi_m f-f\|_w
\le
C_wC_L\,\delta_m\|f\|_s.
\]
For the second term, \textbf{(P2)} applied to \(L_nf\in B_s\) gives
\[
\|\pi_mL_nf-L_nf\|_w\le \delta_m\|L_nf\|_s,
\]
and therefore by \textbf{(LY)},
\[
\|\pi_mL_nf-L_nf\|_w
\le
\delta_m\bigl(\lambda\|f\|_s+B\|f\|_w\bigr).
\]
Summing the two bounds yields
\[
\|M_n^{(m)} f-L_nf\|_w
\le
\delta_m\Bigl[(C_wC_L+\lambda)\|f\|_s+B\|f\|_w\Bigr].
\]
Since \(B_s\hookrightarrow B_w\), this implies \(\|M_n^{(m)}-L_n\|_{B_s\to B_w}\to0\).
\end{proof}

We can finally lift the estimate to the global operators on sequence spaces. Recall that
\[
\mathcal B_s:=\ell^\infty(\mathbb Z;B_s),
\qquad
\mathcal B_w:=\ell^\infty(\mathbb Z;B_w),
\]
with norms
\[
\|\boldsymbol\mu\|_{\mathcal B_s}:=\sup_{n\in\mathbb Z}\|\mu_n\|_s,
\qquad
\|\boldsymbol\mu\|_{\mathcal B_w}:=\sup_{n\in\mathbb Z}\|\mu_n\|_w.
\]

The global operator associated with \((L_n)\) is
\[
\mathbb F:\mathcal B_s\to\mathcal B_s,
\qquad
(\mathbb F\boldsymbol\mu)_n:=L_{n-1}\mu_{n-1},
\qquad n\in\mathbb Z.
\]
Similarly, the global operator associated with \((M_n^{(m)})\) is
\[
\mathbb F_m:\mathcal B_s\to\mathcal B_s,
\qquad
(\mathbb F_m\boldsymbol\mu)_n:=M_{n-1}^{(m)}\mu_{n-1},
\qquad n\in\mathbb Z.
\]
We also define the global projection as
\[
\Pi_m:\mathcal B_w\to\mathcal B_w,
\qquad
(\Pi_m\boldsymbol\mu)_n:=\pi_m\mu_n.
\]

\begin{proposition}[Global-map estimates]
\label{prop:abstract_global_projection}
Under the assumptions of Proposition~\ref{cor:abstract_sequential_projection}, the following hold.

\begin{enumerate}
\item[\textup{(i)}] The global projection \(\Pi_m\) is bounded on \(\mathcal B_w\) and on \(\mathcal B_s\), with
\[
\|\Pi_m\boldsymbol\mu\|_{\mathcal B_w}\le C_w\|\boldsymbol\mu\|_{\mathcal B_w},
\qquad
\|\Pi_m\boldsymbol\mu\|_{\mathcal B_s}\le C_s\|\boldsymbol\mu\|_{\mathcal B_s},
\]
and
\[
\|\Pi_m\boldsymbol\mu-\boldsymbol\mu\|_{\mathcal B_w}
\le
\delta_m\|\boldsymbol\mu\|_{\mathcal B_s}
\qquad \forall \boldsymbol\mu\in\mathcal B_s.
\]

\item[\textup{(ii)}] The global discretized operator can be written as
\[
\mathbb F_m=\Pi_m\,\mathbb F\,\Pi_m.
\]

\item[\textup{(iii)}] The operator \(\mathbb F\) satisfies the strong estimate
\[
\|\mathbb F\boldsymbol\mu\|_{\mathcal B_s}
\le
\lambda \|\boldsymbol\mu\|_{\mathcal B_s}+B\|\boldsymbol\mu\|_{\mathcal B_w}
\qquad \forall \boldsymbol\mu\in\mathcal B_s,
\]
and the weak estimate
\[
\|\mathbb F\boldsymbol\mu\|_{\mathcal B_w}
\le
C_L\|\boldsymbol\mu\|_{\mathcal B_w}
\qquad \forall \boldsymbol\mu\in\mathcal B_w.
\]

\item[\textup{(iv)}] The operator \(\mathbb F_m\) satisfies
\[
\|\mathbb F_m\boldsymbol\mu\|_{\mathcal B_s}
\le
C_s\lambda C_s\,\|\boldsymbol\mu\|_{\mathcal B_s}
+
C_sBC_w\,\|\boldsymbol\mu\|_{\mathcal B_w}
\qquad \forall \boldsymbol\mu\in\mathcal B_s,
\]
and
\[
\|\mathbb F_m-\mathbb F\|_{\mathcal B_s\to\mathcal B_w}
\longrightarrow 0
\qquad\text{as }m\to\infty.
\]
More precisely,
\[
\|(\mathbb F_m-\mathbb F)\boldsymbol\mu\|_{\mathcal B_w}
\le
\delta_m\Bigl[(C_wC_L+\lambda)\|\boldsymbol\mu\|_{\mathcal B_s}
+
B\|\boldsymbol\mu\|_{\mathcal B_w}\Bigr]
\qquad \forall \boldsymbol\mu\in\mathcal B_s.
\]
\end{enumerate}
\end{proposition}

\begin{proof}
Part \textup{(i)} follows immediately from the corresponding properties of \(\pi_m\), taking the supremum over coordinates.

For part \textup{(ii)}, let \(\boldsymbol\mu\in\mathcal B_s\). Then
\[
(\Pi_m\mathbb F\Pi_m\boldsymbol\mu)_n
=
\pi_m L_{n-1}(\pi_m\mu_{n-1})
=
M_{n-1}^{(m)}\mu_{n-1}
=
(\mathbb F_m\boldsymbol\mu)_n.
\]

For part \textup{(iii)},
\[
\|(\mathbb F\boldsymbol\mu)_n\|_s
=
\|L_{n-1}\mu_{n-1}\|_s
\le
\lambda \|\mu_{n-1}\|_s+B\|\mu_{n-1}\|_w,
\]
and taking the supremum over \(n\) gives
\[
\|\mathbb F\boldsymbol\mu\|_{\mathcal B_s}
\le
\lambda \|\boldsymbol\mu\|_{\mathcal B_s}+B\|\boldsymbol\mu\|_{\mathcal B_w}.
\]
The weak estimate is similar.

For part \textup{(iv)}, the Lasota--Yorke estimate follows from Proposition~\ref{cor:abstract_sequential_projection} by taking the supremum over \(n\). Likewise,
\[
\|(\mathbb F_m\boldsymbol\mu-\mathbb F\boldsymbol\mu)_n\|_w
=
\|M_{n-1}^{(m)}\mu_{n-1}-L_{n-1}\mu_{n-1}\|_w
\]
is bounded by
\[
\delta_m\Bigl[(C_wC_L+\lambda)\|\mu_{n-1}\|_s+B\|\mu_{n-1}\|_w\Bigr],
\]
again by Proposition~\ref{cor:abstract_sequential_projection}. Taking the supremum over \(n\) gives the claimed estimate and hence
\[
\|\mathbb F_m-\mathbb F\|_{\mathcal B_s\to\mathcal B_w}\to0.
\]
\end{proof}

\begin{remark}
The classical Ulam discretization in the \(BV\)–\(L^1\) setting {\color{black} (where BV stands for Bounded Variation)} is a particular case of the abstract scheme above, with \(C_w=C_s=1\) and \(\delta_m\) equal to the mesh size of the partition. (see Appendix \ref{subsec:ulam_bv_l1})
In that case, the discretized operators satisfy the same Lasota--Yorke inequality as the original ones. The point of the present formulation is that the same mechanism applies, with minimal changes, to more general finite-rank approximation schemes.
\end{remark}

We are now in a position to combine the results of this section and recover, in the present nonautonomous setting, the same optimal approximation rate for the classical Ulam method known in the autonomoous case (see \cite{BoseMurray2001}).

\begin{proposition}[Approximation rate for Ulam-like discretizations]
\label{prop:projection_delta_log_delta}
Let \((L_n)_{n\in\mathbb Z}\) be a sequence of Markov operators acting on
\(
(B_s,\|\cdot\|_s)\subset (B_w,\|\cdot\|_w).
\)
Assume that \((L_n)\) satisfies Assumption ~\ref{ass:elom} and the uniform Lasota--Yorke inequality \textup{(LY),(W)}.

Let \((\pi_m)_{m\ge1}\) be a family of finite-rank projections satisfying \textup{(P1)} and \textup{(P2)}, with
\[
C_w\le 1,
\qquad
C_s\le 1.
\]
For each \(m\), define
\[
M_n^{(m)}:=\pi_mL_n\pi_m,
\qquad n\in\mathbb Z.
\]
Let
\[
\boldsymbol h=(h_n)_{n\in\mathbb Z}\in \mathcal P_s^{\mathbb Z}
\]
be the equivariant family of the original system, and let
\[
\boldsymbol h^{(m)}=(h_n^{(m)})_{n\in\mathbb Z}\in \mathcal P_s^{\mathbb Z}
\]
be an equivariant family of the discretized system.

Then there exist \(m_0\in\mathbb N\) and \(C>0\) such that, for every \(m\ge m_0\),
\[
\sup_{n\in\mathbb Z}\|h_n^{(m)}-h_n\|_w
\le
C\,\delta_m\,|\log \delta_m|
\]
where $\delta_m$ is as in assumption \textup{(P2)}. 
\end{proposition}

\begin{proof}
 Assumption~\ref{ass:Ms} is automatically satisfied in the present setting, since the discretized operators inherit the same Lasota--Yorke inequality by Proposition~\ref{cor:abstract_sequential_projection}, and here \(C_w\le1\), \(C_s\le1\).

Still by Proposition~\ref{cor:abstract_sequential_projection}, the discretized operators satisfy the mixed strong-to-weak estimate
\[
\sup_{n\in\mathbb Z}\|M_n^{(m)}-L_n\|_{B_s\to B_w}
\le C_1\,\delta_m
\]
for some constant \(C_1>0\) independent of \(m\). 
 Hence both the original system and the discretized one satisfy the hypotheses of Theorem~\ref{thm:dlogd}.

Applying Theorem~\ref{thm:dlogd} with perturbation size
\[
\tau_m:=\sup_{n\in\mathbb Z}\|M_n^{(m)}-L_n\|_{B_s\to B_w}\le C_1\,\delta_m,
\]
we obtain
\[
\|\boldsymbol h^{(m)}-\boldsymbol h\|_{\mathcal B_w}
\le
C_2\,\tau_m\,|\log \tau_m|
\]
for some constant \(C_2>0\) independent of \(m\) proving the claim.
\end{proof}

\section{Memory loss for projection-based discretizations}
\label{subsec:ulam_loss_of_memory}

In this section we show that, under suitable assumptions, a sufficiently fine finite-dimensional reduction inherits exponential memory loss from the original system, with uniform constants for all sufficiently large resolutions.
This result will be used as a tool in Section \ref{subsec:random_systems_satisfying_assumptions}, but it is also of independent interest, as it shows that a suitable finite element reduction of a system having exponential memory loss preserves this important feature of the system.

Let \((L_n)_{n\in\mathbb Z}\) be a sequence of Markov operators
\(
L_n:B_w\to B_w,L_n:B_s\to B_s, \) projections \( \ (\pi_m)_{m\ge1}
\)
and discretized operators
\(
   M_n^{(m)}:=\pi_mL_n\pi_m,
\qquad n\in\mathbb Z
\)
as above satisfying uniformly (P1)-(W), furthermore suppose that

\medskip

\noindent
\textbf{(M)} The projections $\pi_m$ are mass preserving (i.e. $[\pi_m (f)](X)=f(X)$), and 

$C_w,C_s,C_L\leq 1$.

\medskip

The additional assumption (M) ensures that both the original and the discretized operators preserve the zero-mass subspaces, and that the weak norm is not enlarged by the projection or by one step of the dynamics.
Under these assumptions, both the original and the discretized operators preserve the zero-mass subspaces \(V_s\) and \(V_w\), and satisfy a common one-step Lasota--Yorke inequality.

There exist constants \(B\ge 1\) and \(\lambda_1\in[0,1)\) such that, for every \(n\in\mathbb Z\), every \(m\ge1\), and every \(f\in B_s\),
\begin{equation}\label{eq:common_LY_original}
\|L_n f\|_w\le \|f\|_w,
\qquad
\|L_n f\|_s\le \lambda_1\|f\|_s+B\|f\|_w,
\end{equation}
and
\begin{equation}\label{eq:common_LY_discretized}
\|M_n^{(m)} f\|_w\le \|f\|_w,
\qquad
\|M_n^{(m)} f\|_s\le \lambda_1\|f\|_s+B\|f\|_w.
\end{equation}

For \(j\in\mathbb Z\) and \(n\ge1\), let us denote
\[
L^{(j,j+n-1)}:=L_{j+n-1}\circ\cdots\circ L_j,
\qquad
M^{(m),(j,j+n-1)}:=M_{j+n-1}^{(m)}\circ\cdots\circ M_j^{(m)}.
\]

We first record the corresponding Lasota--Yorke estimate for finite sequential compositions.

\begin{lemma}[Sequential Lasota--Yorke estimate]
\label{lem:sequential_LY_memory}
Assume \eqref{eq:common_LY_original}. Then, for every \(j\in\mathbb Z\), \(n\ge1\), and \(f\in B_s\),
\[
\|L^{(j,j+n-1)}f\|_w\le \|f\|_w
\]
and
\[
\|L^{(j,j+n-1)}f\|_s
\le
\lambda_1^n\|f\|_s+\frac{B}{1-\lambda_1}\|f\|_w.
\]
The same estimate holds for \(M^{(m),(j,j+n-1)}\), uniformly in \(m\).
\end{lemma}

\begin{proof}
The weak estimate follows immediately by iteration of \(\|L_n f\|_w\le \|f\|_w\). For the strong estimate, iterating the one-step Lasota--Yorke inequality gives
\[
\|L^{(j,j+n-1)}f\|_s
\le
\lambda_1^n\|f\|_s
+
B\sum_{k=0}^{n-1}\lambda_1^k\|f\|_w
\le
\lambda_1^n\|f\|_s+\frac{B}{1-\lambda_1}\|f\|_w.
\]
The proof for \(M^{(m),(j,j+n-1)}\) is identical, using \eqref{eq:common_LY_discretized}.
\end{proof}

We next estimate the distance between an \(n\)-step discretized block and the corresponding original block.
\begin{lemma}[Distance between original and discretized blocks]
\label{lem:ulam_block_closeness}
Under the assumptions stated at beginning of this section. Let
\[
\varepsilon_m:=\sup_{n\in\mathbb Z}\|M_n^{(m)}-L_n\|_{B_s\to B_w}.
\]
Then, for every \(j\in\mathbb Z\), every \(n\ge1\), and every \(g\in B_s\),
\begin{equation}\label{eq:block_closeness_estimate}
\|M^{(m),(j,j+n-1)}g-L^{(j,j+n-1)}g\|_w
\le
\varepsilon_m\left(\frac{1}{1-\lambda_1}\|g\|_s+n\frac{B}{1-\lambda_1}\|g\|_w\right).
\end{equation}
\end{lemma}

\begin{proof}
We argue by induction on \(n\). For \(n=1\), the claim is immediate from the definition of \(\varepsilon_m\), since
\[
\|M_j^{(m)}g-L_jg\|_w\le \varepsilon_m\|g\|_s.
\]

Assume now that the estimate holds for \(n-1\). Then
\[
\begin{aligned}
&M^{(m),(j,j+n-1)}g-L^{(j,j+n-1)}g\\
&\qquad=
M_{j+n-1}^{(m)}M^{(m),(j,j+n-2)}g
-
L_{j+n-1}L^{(j,j+n-2)}g.
\end{aligned}
\]
Adding and subtracting \(M_{j+n-1}^{(m)}L^{(j,j+n-2)}g\), and using the weak contraction of \(M_{j+n-1}^{(m)}\), we obtain
\[
\begin{aligned}
\|M^{(m),(j,j+n-1)}g-L^{(j,j+n-1)}g\|_w
&\le
\|M^{(m),(j,j+n-2)}g-L^{(j,j+n-2)}g\|_w\\
&\quad+
\|(M_{j+n-1}^{(m)}-L_{j+n-1})L^{(j,j+n-2)}g\|_w.
\end{aligned}
\]
By the induction hypothesis,
\[
\|M^{(m),(j,j+n-2)}g-L^{(j,j+n-2)}g\|_w
\le
\varepsilon_m\left(\frac{1}{1-\lambda_1}\|g\|_s+(n-1)\frac{B}{1-\lambda_1}\|g\|_w\right).
\]
For the second term, using the definition of \(\varepsilon_m\) and Lemma~\ref{lem:sequential_LY_memory},
\[
\begin{aligned}
\|(M_{j+n-1}^{(m)}-L_{j+n-1})L^{(j,j+n-2)}g\|_w
&\le
\varepsilon_m\,\|L^{(j,j+n-2)}g\|_s\\
&\le
\varepsilon_m\left(\lambda_1^{n-1}\|g\|_s+\frac{B}{1-\lambda_1}\|g\|_w\right).
\end{aligned}
\]
Combining the two bounds gives
\[
\|M^{(m),(j,j+n-1)}g-L^{(j,j+n-1)}g\|_w
\le
\varepsilon_m\left(C_{n-1}\|g\|_s+n\frac{B}{1-\lambda_1}\|g\|_w\right),
\]
where \(C_n=C_{n-1}+\lambda_1^{n-1}\). Since
\[
\sum_{k\ge0}\lambda_1^k=\frac{1}{1-\lambda_1},
\]
we have \(C_n\le \frac{1}{1-\lambda_1}\), which proves \eqref{eq:block_closeness_estimate}.
\end{proof}

{\color{black} Next result shows that under suitable additional regularity and closeness assumptions, exponential  memory loss } for the original system implies exponential memory loss for the discretized one.

\begin{proposition}[Memory loss for the discretized system]
\label{prop:ulam_inherits_memory}
Under the assumptions stated at beginning of this section. Suppose there is an integer \(M\ge1\) such that
\begin{equation}\label{eq:M_choice_lambda}
\lambda_1^M\le \frac{1}{10\left(\frac{B}{1-\lambda_1}+1\right)}
\end{equation}
and
\begin{equation}\label{eq:original_block_weak_small}
\|L^{(j,j+M-1)}v\|_w
\le
\frac{1-\lambda_1}{10B}\,\|v\|_s
\qquad
\forall j\in\mathbb Z,\ \forall v\in V_s.
\end{equation}
Assume finally that \(m\) is such that
\begin{equation}\label{eq:epsilon_small_memory}
\varepsilon_m
=
\sup_{n\in\mathbb Z}\|M_n^{(m)}-L_n\|_{B_s\to B_w}
\le
\frac{7(1-\lambda_1)^2}{10MB\left(\frac{1}{1-\lambda_1}+B\right)}.
\end{equation}
Then, for every \(j\in\mathbb Z\) and every \(g\in V_s\),
\begin{equation}\label{eq:2M_block_contraction}
\|M^{(m),(j,j+2M-1)}g\|_s\le \frac{9}{10}\|g\|_s.
\end{equation}
In particular, the discretized system has  uniform exponential memory loss provided $m$ is large enough: there exist \( C>0\) and \(\rho\in(0,1)\), such that for each $m $ satisfying \eqref{eq:epsilon_small_memory}
\[
\|M^{(m),(j,j+n-1)}g\|_s
\le
C\rho^n\|g\|_s
\qquad
\forall j\in\mathbb Z,\ \forall n\ge1,\ \forall g\in V_s.
\]
\end{proposition}

\begin{proof}
Let \(g\in V_s\). Applying Lemma~\ref{lem:ulam_block_closeness} with \(n=M\), and using \(\|g\|_w\le \|g\|_s\), we obtain
\[
\begin{aligned}
\|M^{(m),(j,j+M-1)}g-L^{(j,j+M-1)}g\|_w
&\le
\varepsilon_m\left(\frac{1}{1-\lambda_1}+M\frac{B}{1-\lambda_1}\right)\|g\|_s\\
&\le
\frac{7(1-\lambda_1)}{10B}\|g\|_s,
\end{aligned}
\]
by \eqref{eq:epsilon_small_memory}. Hence, by \eqref{eq:original_block_weak_small},
\[
\begin{aligned}
\|M^{(m),(j,j+M-1)}g\|_w
&\le
\|L^{(j,j+M-1)}g\|_w
+
\|M^{(m),(j,j+M-1)}g-L^{(j,j+M-1)}g\|_w\\
&\le
\frac{1-\lambda_1}{10B}\|g\|_s+\frac{7(1-\lambda_1)}{10B}\|g\|_s\\
&=
\frac{8(1-\lambda_1)}{10B}\|g\|_s.
\end{aligned}
\]
We now apply the sequential Lasota--Yorke estimate of Lemma~\ref{lem:sequential_LY_memory} to the second block of length \(M\), with input \(h:=M^{(m),(j,j+M-1)}g\in V_s\). This gives
\[
\|M^{(m),(j+M,j+2M-1)}h\|_s
\le
\lambda_1^M\|h\|_s+\frac{B}{1-\lambda_1}\|h\|_w.
\]
Using again Lemma~\ref{lem:sequential_LY_memory},
\[
\|h\|_s
=
\|M^{(m),(j,j+M-1)}g\|_s
\le
\lambda_1^M\|g\|_s+\frac{B}{1-\lambda_1}\|g\|_w
\le
\left(\frac{B}{1-\lambda_1}+1\right)\|g\|_s.
\]
Therefore, by \eqref{eq:M_choice_lambda},
\[
\lambda_1^M\|h\|_s
\le
\frac{1}{10}\|g\|_s.
\]
Moreover, by the weak estimate just proved,
\[
\frac{B}{1-\lambda_1}\|h\|_w
\le
\frac{B}{1-\lambda_1}\cdot \frac{8(1-\lambda_1)}{10B}\|g\|_s
=
\frac{8}{10}\|g\|_s.
\]
Combining the two bounds yields
\[
\|M^{(m),(j,j+2M-1)}g\|_s
\le
\frac{1}{10}\|g\|_s+\frac{8}{10}\|g\|_s
=
\frac{9}{10}\|g\|_s,
\]
which proves \eqref{eq:2M_block_contraction}.

Finally, iterating the \(2M\)-step contraction blockwise, one obtains
\[
\|M^{(m),(j,j+2kM-1)}g\|_s\le \left(\frac{9}{10}\right)^k\|g\|_s.
\]
The general case \(n\ge1\) follows by decomposing \(n\) into blocks of length \(2M\) plus a remainder, and using the uniform boundedness of finite compositions provided by Lemma~\ref{lem:sequential_LY_memory}.
\end{proof}

{\color{black}
\begin{remark}
We observe that assumption \eqref{eq:original_block_weak_small} is an immediate consequence of strong loss of memory for the operators \(L_n\). Accordingly, Proposition~\ref{prop:ulam_inherits_memory} can be interpreted as a rigorous formulation of the principle that, in the presence of a common one-step Lasota-Yorke inequality,  exponential loss of memory is preserved under sufficiently fine discretization. 
\end{remark}
}

\section{Linear response and Ulam approximation}\label{S4}

In this section we show that under suitable additional assumptions on the non-autonomous system - i.e. supposing that the transfer operator is  in some sense regularizing, like it happens for systems with smoothly distributed noise - the linear response of the system can be approximated by the linear response of its Ulam discretization.
We remark that a result of this kind is useful and original also in the autonomous case.

In order to introduce the main ingredients whose approximation we are going to address in this section, we briefly recall the linear response framework in the present sequential nonautonomous setting; full details can be found in \cite{GL26}. In particular, the abstract linear response result based on the global-map resolvent is given in \cite[Theorem~11]{GL26}; see also \cite[Remark~12]{GL26} for the corresponding Neumann-series interpretation.

Let \((L_n^\varepsilon)_{n\in\mathbb Z}\) be a family of Markov operators on \(B_s\subset B_w\), depending on a parameter \(\varepsilon\), and suppose that for each \(\varepsilon\) there exists an equivariant family
\[
\boldsymbol h^\varepsilon=(h_n^\varepsilon)_{n\in\mathbb Z}\in \mathcal P_s^{\mathbb Z},
\qquad
h_{n+1}^\varepsilon=L_n^\varepsilon h_n^\varepsilon
\quad \forall n\in\mathbb Z.
\]
At \(\varepsilon=0\), write
\[
L_n:=L_n^0,
\qquad
h_n:=h_n^0,
\qquad
\boldsymbol h:=(h_n)_{n\in\mathbb Z}.
\]

Introduce the global operator
\[
\mathbb F_\varepsilon:\mathcal B_s\to\mathcal B_s,
\qquad
(\mathbb F_\varepsilon\boldsymbol {\color{black} h})_n:=L_{n-1}^\varepsilon {\color{black} h}_{n-1},
\]
so that the equivariance relation becomes
\[
\mathbb F_\varepsilon(\boldsymbol h^\varepsilon)=\boldsymbol h^\varepsilon.
\]
Assume that \(\varepsilon\mapsto L_n^\varepsilon\) and \(\varepsilon\mapsto \boldsymbol h^\varepsilon\) are differentiable at \(\varepsilon=0\), and denote
\[
\dot L_n:=\left.\partial_\varepsilon L_n^\varepsilon\right|_{\varepsilon=0},
\qquad
\dot{\boldsymbol h}:=\left.\partial_\varepsilon \boldsymbol h^\varepsilon\right|_{\varepsilon=0}.
\]
{\color{black} (Precise differentiability assumptions will be given in next subsection). } 
Differentiating the equivariance relation yields
\[
(I-\mathbb F_0)\dot{\boldsymbol h}=\boldsymbol{\dot{L}h},
\qquad
(\boldsymbol{\dot{L}h})_n:=\dot L_{n-1}h_{n-1}.
\]
Thus, provided \(I-\mathbb F_0\) is invertible on the relevant zero-mass sequence space
\begin{equation}\label{LR}
    \dot{\boldsymbol h}=(I-\mathbb F_0)^{-1}\boldsymbol{\dot{L}h}.
\end{equation}
Whenever the Neumann series converges, this may be written as
\begin{equation}\label{r3}
\dot{\boldsymbol h}=\sum_{k\ge0}\mathbb F_0^k\boldsymbol{\dot{L}h}.
\end{equation}
In \cite{GL26} the reader can find the correct assumptions and the mathematical setting in which \eqref{r3} makes sense
and converge.

The same strategy can be applied  to the Ulam discretization. Let \(\pi_m:B_w\to B_w\) be the projection associated with the classical Ulam discretization (see Appendix \ref{subsec:ulam_bv_l1} for precise definitions) of resolution \(m\). Let us follow the construction to fix the notation we will use in the following. Let us set
\[
M_n^\varepsilon:=\pi_m L_n^\varepsilon \pi_m,
\qquad
(\mathbb F_m^\varepsilon\boldsymbol {\color{black}h})_n:=M_{n-1}^\varepsilon{\color{black}h}_{n-1}.
\]
Assume that for each \(\varepsilon\) there exists an equivariant family
\[
\boldsymbol h^{\varepsilon,m}=(h_n^{\varepsilon,m})_{n\in\mathbb Z},
\qquad
h_{n+1}^{\varepsilon,m}=M_n^\varepsilon h_n^{\varepsilon,m}.
\]
Writing
\[
M_n:=M_n^0,
\qquad
\boldsymbol h^m:=(h_n^m)_{n\in\mathbb Z},
\qquad
\dot M_n:=\left.\partial_\varepsilon M_n^\varepsilon\right|_{\varepsilon=0},
\]
and assuming differentiability of \(\varepsilon\mapsto \boldsymbol h^{\varepsilon,m}\) at \(\varepsilon=0\), with derivative
\[
\hat{\boldsymbol h}^m:=\left.\partial_\varepsilon \boldsymbol h^{\varepsilon,m}\right|_{\varepsilon=0},
\]
one obtains
\[
(I-\mathbb F_m^0)\dot{\boldsymbol h}^m=\boldsymbol{ \dot Mh^m},
\qquad
(\boldsymbol{ \dot Mh^m})_n:=\dot M_{n-1}h_{n-1}^m,
\]
and hence
\[
\dot{\boldsymbol h}^m=(I-\mathbb F_m^0)^{-1}\boldsymbol{ \dot Mh^m}.
\]

{\color{black}\begin{remark}\label{22}[Differentiability of the Ulam discretizations.]
By the definition of
$M_n^\varepsilon:=\pi_m L_n^\varepsilon\pi_m,$
one sees that as \(\varepsilon\to0\), then
\[
\frac{M_n^\varepsilon-M_n^0}{\varepsilon}
=
\pi_m
\frac{L_n^\varepsilon-L_n^0}{\varepsilon}
\pi_m
\longrightarrow
\pi_m\dot L_n\pi_m.
\]
Hence \(\varepsilon\mapsto M_n^\varepsilon\) is differentiable at
\(\varepsilon=0\), with
\[
\dot M_n=\pi_m\dot L_n\pi_m.
\]
If the convergence for \(L_n^\varepsilon\) is uniform in \(n\), in some suitable norm, then the same
is true for the discretized family \(M_n^\varepsilon\).
\end{remark}}

Details on the response formula for this non homogeneous finite Markov chains case are given in \cite{Lucarini2026,GL26}.

The term $\dot{\boldsymbol h}^m$ is what one can compute as linear response of the Ulam discretization of our system.  
In what follows we will show that, under suitable assumptions, the linear response  $\dot{\boldsymbol h}^m$ of the reduced system converges in the weak norm to the projection $\Pi_m \dot{\boldsymbol h}$ of the linear response of the original system.

The comparison between \(\dot{\boldsymbol h}\) and \(\dot{\boldsymbol h}^m\) boils down to two tasks: i) controlling the source-term error \(\boldsymbol{ \dot Mh^m}-\Pi_m\boldsymbol{ \dot Lh}\); and ii) controlling the discrepancy between the projected resolvents
\[
(I-\mathbb F_m^0)^{-1}\Pi_m
\quad\text{and}\quad
\Pi_m(I-\mathbb F_0)^{-1}.
\]
In the following these two ingredients are handled separately.

\subsection{Approximation of the source term, resolvent, and linear response}
In this section we will continue to profit from the compact notation of the global maps acting on sequence spaces. 
Recall that $\mathcal B_w, B_s$ denote the sequence spaces on which the global operator  $\mathbb F$ acts, and $\mathcal V_w,\mathcal V_s$ denotes the zero average spaces. Recall that $ \Pi_m $ and $\mathbb F_m$ denote the global projection and the global discretized operator.

Assume the system satisfy Assumption \ref{ass:elom} (exponential strong memory loss) and the standing hypotheses from the beginning of Section \ref{subsec:ulam_loss_of_memory}
and in particular that the projections $\pi_m$ satisfy (M), (P1), (P2), furthermore, we assume

\medskip

\noindent
\textbf{(H1)} For each \(n\in\mathbb Z\), \(m\in\mathbb N\) the limits
\[
\psi_{n}:=\lim_{\varepsilon\to0}
\frac{L_{n-1}^\varepsilon h_{n-1}-L_{n-1} h_{n-1}}{\varepsilon}
\quad
\psi_{n}^m:=\lim_{\varepsilon\to0}
\frac{M_{n-1}^\varepsilon h^m_{n-1}-M_{n-1} h^m_{n-1}}{\varepsilon}
\]
exists in \(B_s\), and
\[
\sup_{n\in\mathbb Z}\|\psi_n\|_s<\infty
\quad
\sup_{n\in\mathbb Z}\|\psi_n^m\|_s<\infty
.
\]

\noindent
\textbf{(H2)}  The map
\[
\varepsilon\mapsto L_n^\varepsilon\in\mathcal L(B_w,B_w)
\]
is {\color{black} uniformly} differentiable at \(\varepsilon=0\) in operator norm, that is, there exists
\[
\dot L_n\in\mathcal L(B_w,B_w)
\]
such that
\[
\left\|
L_n^\varepsilon-L_n^0-\varepsilon \dot L_n
\right\|_{B_w\to B_w}
=o(\varepsilon)
\qquad\text{as }\varepsilon\to0
\]
{\color{black} uniformly, as $n$ ranges in \(\mathbb Z\).}
Moreover,
\[
\sup_{n\in\mathbb Z}\|\dot L_n\|_{B_w\to B_w}<\infty.
\]
\noindent
\textbf{(H3)} There exists \(C_{\mathrm{reg}}>0\) such that
\[
\|L_n f\|_s\le C_{\mathrm{reg}}\|f\|_w
\qquad
\forall n\in\mathbb Z,\ \forall f\in B_w.
\]

\medskip


From the assumptions, by Proposition \ref{prop:ulam_inherits_memory}, and (H3), it directly follows 

\begin{lemma}\label{lem:H5}
The operators \(\mathbb F\), \(\mathbb F_m\), and \(\Pi_m\) preserve \(\mathcal V_w\), and there exist \(C^1_*,C^2_*>0\) and \(\rho_1,\rho_2\in(0,1)\) such that
\[
\|\mathbb F^k\boldsymbol u\|_{\mathcal B_w}
\le C_*^1\rho_1^k\|\boldsymbol u\|_{\mathcal B_w},
\qquad
\|\mathbb F_m^k\boldsymbol u\|_{\mathcal B_w}
\le C^1_*\rho_1^k\|\boldsymbol u\|_{\mathcal B_w},
\]
\[
\|\mathbb F^k\boldsymbol u\|_{\mathcal B_s}
\le C_*^2\rho_2^k\|\boldsymbol u\|_{\mathcal B_s},
\qquad
\|\mathbb F_m^k\boldsymbol u\|_{\mathcal B_s}
\le C^2_*\rho_2^k\|\boldsymbol u\|_{\mathcal B_s},
\]
for every \(k\ge0\), every \(m\), and every \(\boldsymbol u\) respectively \(\in\mathcal V_w\) or \(\in\mathcal V_s\) .
  
\end{lemma}

In the following lemma we start addressing the approximation of the linear response by its two main ingredients, i.e. the derivative operators defined in {\bf (H2)}, and the resolvents, $(I-\mathbb F)^{-1}, (I-\mathbb F_m)^{-1}$ whose existence on $\mathcal{V}_w$, $\mathcal{V}_s$ is granted by Lemma \ref{lem:H5}.

\begin{lemma}[Approximation of source terms and resolvents]\label{lem:source_resolvent_approx}
Assume the standing hypotheses from the beginning of Section \ref{subsec:ulam_loss_of_memory}  and \textbf{(H1)}--\textbf{(H3)}. Let
\(
\boldsymbol h
\)
be the equivariant family for \((L_n)_{n\in\mathbb Z}\), and let
\(
\boldsymbol h^m
\)
be an equivariant family for \((M_n)_{n\in\mathbb Z}\).
Then the following hold.

\smallskip

\noindent
\textup{(i)} For every \(n\in\mathbb Z\),
\[
\psi_n^m-\pi_m\psi_n
=
\dot M_{n-1}\bigl(h_{n-1}^m-\pi_m h_{n-1}\bigr)
+
\Bigl(\dot M_{n-1}(\pi_m h_{n-1})-\pi_m\dot L_{n-1}h_{n-1}\Bigr),
\]
and therefore
\[
\|\boldsymbol\psi^m-\Pi_m\boldsymbol\psi\|_{\mathcal B_w}\to0
\]
whenever
\[
\sup_{n\in\mathbb Z}\|h_n^m-\pi_m h_n\|_w\to0.
\]

\smallskip

\noindent
\textup{(ii)} There exists \(C>0\), independent of \(m\), such that
\[
\bigl\|(I-\mathbb F_m)^{-1}\Pi_m-\Pi_m(I-\mathbb F)^{-1}\bigr\|_{\mathcal V_s\to\mathcal B_w}
\le C\tau_m.
\]
In particular,
\[
\bigl\|(I-\mathbb F_m)^{-1}\Pi_m-\Pi_m(I-\mathbb F)^{-1}\bigr\|_{\mathcal V_s\to\mathcal B_w}\to0.
\]
\end{lemma}

\begin{proof}
For part \textup{(i)}, add and subtract \(\dot M_{n-1}(\pi_m h_{n-1})\). Since
\[
\dot M_n=\pi_m\dot L_n\pi_m,
\]
we have
\[
\|\dot M_n\|_{B_w\to B_w}
\le
\|\pi_m\|_{B_w\to B_w}^2\|\dot L_n\|_{B_w\to B_w},
\]
uniformly in \(m\) and \(n\). Moreover,
\[
\dot M_{n-1}(\pi_m h_{n-1})-\pi_m\dot L_{n-1}h_{n-1}
=
\pi_m\dot L_{n-1}(\pi_m h_{n-1}-h_{n-1}),
\]
hence \textbf{(P1),(P2)} and \textbf{(H2)} imply
\[
\sup_{n\in\mathbb Z}
\|\dot M_{n-1}(\pi_m h_{n-1})-\pi_m\dot L_{n-1}h_{n-1}\|_w\to0.
\]
The claimed convergence of \(\boldsymbol\psi^m-\Pi_m\boldsymbol\psi\) follows.

For part \textup{(ii)}, Lemma \ref{lem:H5} implies that on \(\mathcal V_w\)
\[
(I-\mathbb F)^{-1}=\sum_{k\ge0}\mathbb F^k,
\qquad
(I-\mathbb F_m)^{-1}=\sum_{k\ge0}\mathbb F_m^k.
\]
Thus
\[
(I-\mathbb F_m)^{-1}\Pi_m-\Pi_m(I-\mathbb F)^{-1}
=
\sum_{k\ge1}\bigl(\mathbb F_m^k\Pi_m-\Pi_m\mathbb F^k\bigr).
\]
For \(k\ge1\), the telescoping identity gives
\[
\mathbb F_m^k\Pi_m-\Pi_m\mathbb F^k
=
\sum_{j=0}^{k-1}
\mathbb F_m^{k-1-j}(\mathbb F_m\Pi_m-\Pi_m\mathbb F)\mathbb F^j.
\]
Now, for \(\boldsymbol v\in\mathcal B_s\),
\[
\bigl((\mathbb F_m\Pi_m-\Pi_m\mathbb F)\boldsymbol v\bigr)_n
=
\pi_mL_{n-1}\pi_m v_{n-1}-\pi_mL_{n-1}v_{n-1},
\]
so item (b) of Proposition \ref{cor:abstract_sequential_projection} yields
\[
\|(\mathbb F_m\Pi_m-\Pi_m\mathbb F)\boldsymbol v\|_{\mathcal B_w}
\le \tau_m\|\boldsymbol v\|_{\mathcal B_s}.
\]
Combining this with \textbf{(H3)} and Lemma \ref{lem:H5} gives
\[
\|\mathbb F_m^k\Pi_m-\Pi_m\mathbb F^k\|_{\mathcal V_s\to\mathcal B_w}
\le C_0\,k\,\tau_m\,\rho^{k-2}
\qquad (k\ge1)
\]
for some \(C_0>0\) independent of \(m\) and \(k\). Summing in \(k\) yields the claim.
\end{proof}

{\color{black}
Finally, we are ready to prove  a final proposition, establishing a set of assumptions under which the response of a given system can be approximated by the response of its finite elements Ulam-like reduction.
Such assumptions include 
the ones standing from Section \ref{subsec:ulam_loss_of_memory},  \textbf{(H1)}--\textbf{(H3)} and the formal existence of the responses of both the system and its reduction.
In Section \ref{subsec:random_systems_satisfying_assumptions} we show a set of meaningful examples, where all these assumptions are verified.}

\begin{proposition}[Approximation of the linear response]\label{prop:linear_response_approx}
Assume the standing hypotheses from the beginning of Section \ref{subsec:ulam_loss_of_memory}  and \textbf{(H1)}--\textbf{(H3)}.
{\color{black} Suppose}
\[
\hat{\boldsymbol h}:=(I-\mathbb F)^{-1}\boldsymbol\psi,
\qquad
\hat{\boldsymbol h}^m:=(I-\mathbb F_m)^{-1}\boldsymbol\psi^m
\]
{\color{black} are well defined and exist. Then}
\[
\|\hat{\boldsymbol h}^m-\Pi_m\hat{\boldsymbol h}\|_{\mathcal B_w}\to0
\qquad\text{as }m\to\infty.
\]
\end{proposition}

\begin{proof}
We remark that since $h_n^m,h_n\in B_s$ then  \[
\sup_{n\in\mathbb Z}\|h_n^m-\pi_m h_n\|_w\to0.
\]

Decompose
\[
\hat{\boldsymbol h}^m-\Pi_m\hat{\boldsymbol h}
=
(I-\mathbb F_m)^{-1}(\boldsymbol\psi^m-\Pi_m\boldsymbol\psi)
+
\Bigl((I-\mathbb F_m)^{-1}\Pi_m-\Pi_m(I-\mathbb F)^{-1}\Bigr)\boldsymbol\psi.
\]
By  Lemma \ref{lem:H5}, the operator \((I-\mathbb F_m)^{-1}\) is  bounded on \(\mathcal V_w\). The first term therefore tends to zero in the weak norm by Lemma~\ref{lem:source_resolvent_approx}\,\textup{(i)}, while the second tends to zero by Lemma~\ref{lem:source_resolvent_approx}\,\textup{(ii)}, since $\boldsymbol{\psi} \in \mathcal{B}_s$.
\end{proof}

\begin{remark}
The conclusion of Proposition~\ref{prop:linear_response_approx} is stated in the form
\[
\|\hat{\boldsymbol h}^m-\Pi_m\hat{\boldsymbol h}\|_{\mathcal B_w}\to0,
\]
rather than
\[
\|\hat{\boldsymbol h}^m-\hat{\boldsymbol h}\|_{\mathcal B_w}\to0.
\]

However, if, in addition, the linear response of the original system belongs to the strong sequence space, i.e.
\[
\hat{\boldsymbol h}\in\mathcal B_s,
\]
then the approximation property of the projections immediately gives
\[
\|\Pi_m\hat{\boldsymbol h}-\hat{\boldsymbol h}\|_{\mathcal B_w}\to0,
\]
and therefore
\[
\|\hat{\boldsymbol h}^m-\hat{\boldsymbol h}\|_{\mathcal B_w}\to0
\]
by the triangle inequality.

Even when \(\hat{\boldsymbol h}\notin\mathcal B_s\), the statement of Proposition~\ref{prop:linear_response_approx} remains useful. Indeed, it shows that the response of the reduced finite-dimensional model approximates the natural projection of the true response, which is exactly the quantity that can be represented at the chosen discretization level. In this sense, the result still provides a meaningful and computable approximation theorem, even in the absence of additional regularity of \(\hat{\boldsymbol h}\).
\end{remark}

\subsection{Uniformly positive smoothing kernels}
\label{subsec:uniformly_positive_smoothing}

We now describe a concrete class of {\color{black} regularizing} sequential Markov operators for which the abstract framework developed above applies {\color{black} and we can show a strong for of memory loss for these systems (see Theorem \ref{thm:positive_smoothing_kernels})}.
In the next section we will define a general class of random dynamical systems whose transfer operators satisfy the assumptions presented here.
The key ingredients are a uniform Doeblin condition, yielding exponential memory loss in a weak norm, and a uniform smoothing property, yielding regularization into a stronger space.

Let \(X\) be a compact \(C^1\) Riemannian manifold, and let \(m\) be the normalized volume measure on \(X\). We consider finite signed measures on \(X\) which are absolutely continuous with respect to \(m\), and identify such a measure \(\mu\ll m\) with its density \(f=d\mu/dm\). In particular, we set
\[
B_w:=\{\mu\ll m:\ d\mu/dm\in L^1(m)\},
\qquad
\|\mu\|_w:=\left\|\frac{d\mu}{dm}\right\|_{L^1(m)},
\]
and
\[
V_w:=\{\mu\in B_w:\ \mu(X)=0\}.
\]
As a strong space, {\color{black} in this section} one may take
\[
B_s:=\{\mu\ll m:\ d\mu/dm\in C^1(X)\},
\qquad
\|\mu\|_s:=\left\|\frac{d\mu}{dm}\right\|_{C^1(X)},
\]
with corresponding zero-mass space
\[
V_s:=\{\mu\in B_s:\ \mu(X)=0\}.
\]
With this convention, transfer operators may be viewed either as acting on measures or directly on their densities, with no ambiguity in notation.

For each \(n\in\mathbb Z\), let
\[
K_n(x,dy)=k_n(x,y)\,m(dy)
\]
be a Markov kernel on \(X\), and let the associated transfer operator act on densities by
\[
(L_n f)(y):=\int_X f(x)\,k_n(x,y)\,dm(x).
\]
Equivalently, if \(\mu\in B_w\) has density \(f=d\mu/dm\), then \(L_n\mu\) is the absolutely continuous measure with density \(L_n f\). In particular, each \(L_n\) is positive and preserves mass.

We assume the following.

\begin{assumption}[Uniform positivity and regularity]\label{ass:positive_smoothing_kernels}
There exist \(\alpha>0\) and \(C_{\mathrm{reg}}>0\) such that, for every \(n\in\mathbb Z\),
\begin{enumerate}
\item \(k_n(x,y)\ge 0\) and
\[
\int_X k_n(x,y)\,dm(y)=1
\qquad\forall x\in X;
\]
\item
\[
k_n(x,y)\ge \alpha
\qquad\forall x,y\in X;
\]
\item for every \(x\in X\), the map \(y\mapsto k_n(x,y)\) belongs to \(C^1(X)\), and
\[
\sup_{n\in\mathbb Z}\sup_{x\in X}\|k_n(x,\cdot)\|_{C^1(X)}\le C_{\mathrm{reg}}.
\]
\end{enumerate}
\end{assumption}

\begin{theorem}[Exponential memory loss and regularization]\label{thm:positive_smoothing_kernels}
Assume Assumption~\ref{ass:positive_smoothing_kernels}. Then the following hold.

\begin{enumerate}
\item[\textup{(i)}] For every \(j\in\mathbb Z\), \(k\ge1\), and every \(\mu,\nu\in B_w\) with \(\mu(X)=\nu(X)\),
\[
\|L^{(j,j+k-1)}(\mu-\nu)\|_w
\le
(1-\alpha)^k\|\mu-\nu\|_w.
\]
In particular, the sequential system has exponential memory loss on \(V_w\).

\item[\textup{(ii)}] There exists \(R>0\) such that for every \(n\in\mathbb Z\) and every \(\mu\in B_w\),
\[
\|L_n\mu\|_s\le R\|\mu\|_w.
\]
In particular, the system is uniformly regularizing from \(B_w\) to \(B_s\).

\item[\textup{(iii)}] Consequently, for every \(j\in\mathbb Z\), \(k\ge1\), and every \(\mu\in V_w\),
\[
\|L^{(j,j+k)}\mu\|_s
\le
R(1-\alpha)^k\|\mu\|_w.
\]
Thus the system satisfies the weak/strong hypotheses required in the abstract linear response framework developed above.
\end{enumerate}
\end{theorem}

\begin{proof}
Let \(\mu,\nu\in B_w\), and write
\[
h:=\frac{d(\mu-\nu)}{dm}\in L^1(m).
\]
If \(\mu(X)=\nu(X)\), then \(\int_X h\,dm=0\). Since \(L_n\) acts on densities by
\[
(L_n h)(y)=\int_X h(x)\,k_n(x,y)\,dm(x),
\]
the lower bound \(k_n(x,y)\ge\alpha\) implies the Doeblin minorization
\[
K_n(x,\cdot)\ge \alpha\,m(\cdot)
\qquad\forall n\in\mathbb Z,\ \forall x\in X.
\]
Hence, one obtains
\[
\|L_n h\|_{L^1}\le (1-\alpha)\|h\|_{L^1}
\qquad\text{whenever } \int_X h\,dm=0.
\]
Iterating gives
\[
\|L^{(j,j+k-1)}h\|_{L^1}
\le
(1-\alpha)^k\|h\|_{L^1}.
\]
Since \(\|\,\mu-\nu\,\|_w=\|h\|_{L^1}\), this proves \textup{(i)}.

To prove \textup{(ii)}, let \(\mu\in B_w\) and write \(f=d\mu/dm\in L^1(m)\). Then
\[
(L_n f)(y)=\int_X f(x)\,k_n(x,y)\,dm(x),
\]
so
\[
|L_n f(y)|
\le
\int_X |f(x)|\,|k_n(x,y)|\,dm(x)
\le
\Bigl(\sup_{x\in X}\|k_n(x,\cdot)\|_{C^0(X)}\Bigr)\|f\|_{L^1}.
\]
Taking the supremum in \(y\), we obtain
\[
\|L_n f\|_{C^0(X)}\le C_{\mathrm{reg}}\|f\|_{L^1}.
\]

Since \(y\mapsto k_n(x,y)\) is \(C^1\), differentiation under the integral sign yields
\[
D_y(L_n f)(y)=\int_X f(x)\,D_yk_n(x,y)\,dm(x),
\]
and therefore
\[
\|D(L_n f)\|_{C^0(X)}
\le
\Bigl(\sup_{x\in X}\|k_n(x,\cdot)\|_{C^1(X)}\Bigr)\|f\|_{L^1}
\le
C_{\mathrm{reg}}\|f\|_{L^1}.
\]
Thus
\[
\|L_n f\|_{C^1(X)}\le R\|f\|_{L^1}
\]
for some \(R>0\) independent of \(n\). Since \(d(L_n\mu)/dm=L_n f\), this is exactly
\[
\|L_n\mu\|_s\le R\|\mu\|_w,
\]
proving \textup{(ii)}.

Finally, \textup{(iii)} follows by combining \textup{(i)} and \textup{(ii)}. If \(\mu\in V_w\), then
\[
\|L^{(j,j+k)}\mu\|_s
=
\|L_{j+k}(L^{(j,j+k-1)}\mu)\|_s
\le
R\|L^{(j,j+k-1)}\mu\|_w
\le
R(1-\alpha)^k\|\mu\|_w.
\]
This completes the proof.
\end{proof}

\begin{remark}[Kernels from stochastic flows]
Kernels of this type also arise from stochastic differential equations on
manifolds.  Indeed, the time-one map of a stochastic flow does not assign a
single deterministic image to a point \(x\), but rather a transition
probability
\[
K(x,A)=\mathbb P(\Phi_1(x,\omega)\in A).
\]
When this transition probability admits a density \(p_1(x,y)\) with respect
to the Riemannian volume measure \(m\), the associated transfer operator is
\[
(Lf)(y)=\int_X f(x)p_1(x,y)\,dm(x).
\]
Thus \(p_1(x,y)\) is exactly a kernel of the form considered above.

For elliptic diffusion processes on manifolds, the existence of smooth
transition densities is a standard consequence of the regularity theory for
diffusions.
For stochastic differential equations and diffusion semigroups on manifolds,
see \cite[Chapters VIII--IX]{Elworthy1982}.  For Brownian motion,
heat kernels and their positivity on Riemannian manifolds, see \cite[Chapters 3--4]{Hsu2002}.  In particular, on a compact connected
manifold, once the transition density \(p_t(x,y)\) is continuous and strictly
positive for a fixed \(t>0\), compactness gives
\[
\inf_{x,y\in X}p_t(x,y)>0.
\]

\end{remark}

\subsection{A class of random systems satisfying the assumptions}
\label{subsec:random_systems_satisfying_assumptions}

{\color{black} We now describe a broad class of nonautonomous stochastic systems to which the abstract theory applies. Besides providing an example of concrete verification of our assumptions, the class we consider illustrates the range
of models covered by our results. 
In those systems we have a deterministic dynamics driven by maps \(f_n^\varepsilon\)
 that may vary with time, while
the additive, random perturbations may account for unresolved
degrees of freedom or external fluctuations.

For simplicity, we present the construction on the one-dimensional
torus. The same arguments extend in a standard way to multidimensional
tori, or to suitable compact domains, by considering smooth,
strictly positive multidimensional noise densities. In particular,
the resulting class includes commonly used time discretizations of
nonautonomous stochastic differential equations with nondegenerate
additive noise (see \cite[Section 6]{GL26}). Such systems hence arise naturally in models from physics,
the natural sciences, and economics.
}

Let \((\xi_n)_{n\in\mathbb Z}\) be i.i.d.\ random variables on \(\mathbb S^1\) with common density \(q\in C^2(\mathbb S^1)\), and assume that
\begin{equation}\label{eq:noise_positive_c2}
q(y)\ge \alpha>0
\qquad\text{for all }y\in\mathbb S^1.
\end{equation}
Let \((f_n^\varepsilon)_{n\in\mathbb Z}\) be a family of maps \(f_n^\varepsilon:\mathbb S^1\to\mathbb S^1\), and assume that
\begin{equation}\label{eq:fn_eps_c1}
f_n^\varepsilon=f_n+\varepsilon \dot f_n+r_n^\varepsilon,
\qquad
\sup_{n\in\mathbb Z}\|\dot f_n\|_{C^1}<\infty,
\qquad
\sup_{n\in\mathbb Z}\frac{\|r_n^\varepsilon\|_{C^1}}{|\varepsilon|}\xrightarrow[\varepsilon\to0]{}0.
\end{equation}

Consider the random dynamics
\[
X_{n+1}=f_n^\varepsilon(X_n)+\xi_n \quad (\mathrm{mod}\ 1).
\]
The associated annealed transfer operators are
\begin{equation}\label{eq:annealed_L_c2}
(L_n^\varepsilon\varphi)(y)
:=
\int_{\mathbb S^1}\varphi(x)\,q\bigl(y-f_n^\varepsilon(x)\bigr)\,dm(x).
\end{equation}
{\color{black} Since we want to apply the classical Ulam discretization to these systems (see Appendix \ref{subsec:ulam_bv_l1}) we consider as strong and weak spaces}
\[
B_w=L^1(\mathbb S^1),
\qquad
B_s=BV(\mathbb S^1),
\]
with the corresponding zero-mass spaces \(V_w\) and \(V_s\).

The positivity assumption \eqref{eq:noise_positive_c2} implies a uniform Doeblin minorization, since
\[
k_n^\varepsilon(x,y)=q\bigl(y-f_n^\varepsilon(x)\bigr)\ge \alpha
\qquad
\forall n\in\mathbb Z,\ \forall x,y\in\mathbb S^1.
\]
Moreover, the operators are uniformly regularizing from \(L^1\) to \(C^1\) {\color{black}or BV}. Indeed, from \eqref{eq:annealed_L_c2} and \eqref{deriva},
\[
{\color{black}\|L_n^\varepsilon\varphi\|_{BV}}
\le
\|L_n^\varepsilon\varphi\|_{C^1}
\le
\|q\|_{C^1}\|\varphi\|_{L^1}.
\]
{\color{black}This and the choice of the weak and strong space made, shows that the systems satisfy the assumptions $(LY)$, $(W)$ and  $(H3)$.
By theorem \ref{thm:positive_smoothing_kernels}  the systems have exponential loss of memory in \(L^1\), uniformly in \(n\) and \(\varepsilon\) and therefore admits a unique equivariant family
\[
\boldsymbol h^\varepsilon=(h^\varepsilon_n)_{n\in\mathbb Z}\in \mathcal P_s^{\mathbb Z}.
\]}
{\color{black} Because of the Doeblin minorization and the regularization (see item (iii) of Theorem \ref{thm:positive_smoothing_kernels}) the systems also enjoy exponential loss of memory in $C^1$ and for the $BV$ norm.}
In particular they satisfy Assumption \ref{ass:elom}.

Differentiating under the integral sign we get
\[
\partial_y(L_n^\varepsilon\varphi)(y)
=
\int_{\mathbb S^1}\varphi(x)\,q'\bigl(y-f_n^\varepsilon(x)\bigr)\,dm(x),
\]
so that
\begin{equation}\label{deriva}
\|\partial_y(L_n^\varepsilon\varphi)\|_{C^0}
\le
\|q'\|_{C^0}\|\varphi\|_{L^1}.
\end{equation}

We now verify that the family \((L_n^\varepsilon)_{n\in\mathbb Z}\) satisfies the differentiability assumptions {\color{black}(H1),(H2)}. Since \(q'\in L^1(\mathbb S^1)\), translations are \(L^1\)-Lipschitz, and \eqref{eq:fn_eps_c1} we have that \(\varepsilon\mapsto L_n^\varepsilon\) is {\color{black}uniformly } differentiable at \(\varepsilon=0\) in the operator norm of \(L^1\to L^1\), with
\[
(\dot L_n\varphi)(y)
=
-\int_{\mathbb S^1}\varphi(x)\,q'\bigl(y-f_n(x)\bigr)\,\dot f_n(x)\,dm(x),
\]
and
\[
\sup_{n\in\mathbb Z}\|\dot L_n\|_{L^1\to L^1}<\infty.
\]
Hence Assumption~\textbf{(H2)} holds.

Furthermore, since \(q\in C^2\) and \(\boldsymbol h\in\mathcal P_s^{\mathbb Z}\), the source term belongs to the strong space: for each \(n\in\mathbb Z\),
\[
\psi_n:=
\lim_{\varepsilon\to0}
\frac{L_{n-1}^\varepsilon h_{n-1}-L_{n-1}h_{n-1}}{\varepsilon}
\]
exists in \(C^1(\mathbb S^1)\), and is given by
\[
\psi_n(y)
=
-\int_{\mathbb S^1} h_{n-1}(x)\,q'\bigl(y-f_{n-1}(x)\bigr)\,\dot f_{n-1}(x)\,dm(x).
\]
Indeed,
\[
\partial_y\psi_n(y)
=
-\int_{\mathbb S^1} h_{n-1}(x)\,q''\bigl(y-f_{n-1}(x)\bigr)\,\dot f_{n-1}(x)\,dm(x),
\]
so
\[
{\color{black}\sup_{n\in\mathbb Z}\|\psi_n\|_{BV}}\leq
\sup_{n\in\mathbb Z}\|\psi_n\|_{C^1}<\infty.
\]
Thus  assumption { \bf (H1)} is also satisfied.

\noindent{\bf Conclusion.} 
{\color{black}By the results of \cite{GL26}, Section 6 we have that a family of systems as described in this section has linear response, and satsify \eqref{LR}}.
 {\color{black} By this, the computation done in  Remark \ref{22} and the loss of memory shown in Proposition \ref{prop:ulam_inherits_memory} it also follows the differentiability of the reduced system and the existence of its liner response.
  Appendix~\ref{subsec:ulam_bv_l1} shows that the classical Ulam discretization satisfies the abstract assumptions needed for the projections. 
Putting this together with the estimates shown in this section, we have then verified that such family of systems also satisfy all the hypotheses needed to apply Proposition~\ref{prop:projection_delta_log_delta}, concerning the approximation of the equivariant family, and Proposition~\ref{prop:linear_response_approx}, concerning the approximation of the linear response.

It finally } follows that, for this class of systems, the classical Ulam method yields a reliable (asymptotically correct) approximation both of the equivariant family and, for the perturbations introduced above, of the corresponding linear response.

\begin{figure}
    \centering
\includegraphics[width=1.\linewidth]{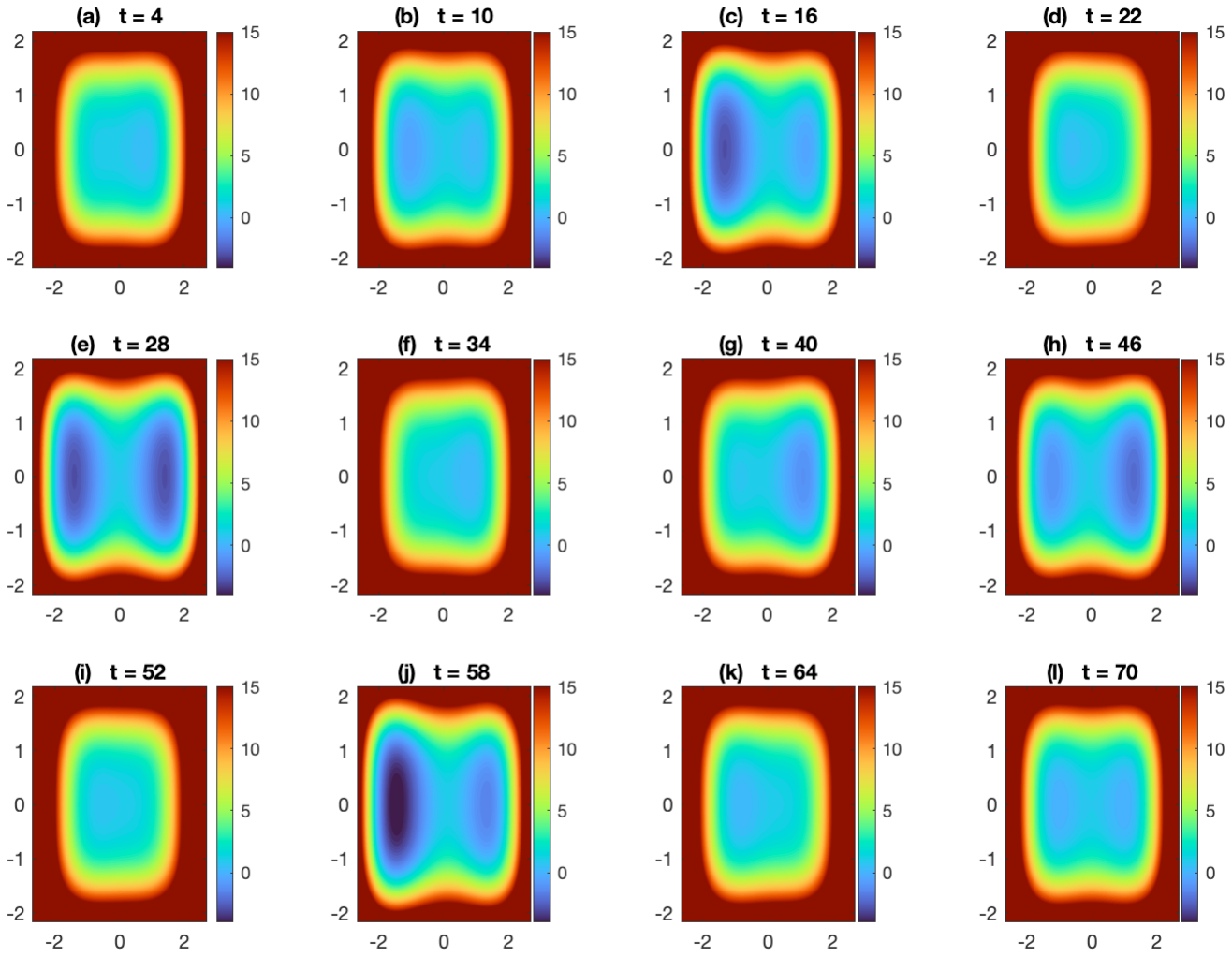}
    \caption{Illustration of the time-dependent potential $V(x,y,t)$. Full movie available online in the Supplementary Information.}
    \label{fig:potentialpanel}
\end{figure}

\section{Numerical Simulations}\label{numericalsim}
{\color{black}We wish to provide a simple yet instructive example of a numerical investigation of the equivariant measure of a time-dependent system and of its response to perturbations.} We consider the  Ito diffusion process described by the following stochastic differential equation:
\begin{align}
    \mathrm{d}x=-\partial_x V(x,y,t)+\sigma \mathrm{d}W_1\nonumber;\\
    \mathrm{d}y=-\partial_y V(x,y,t)+\sigma \mathrm{d}W_2;\label{ito}
\end{align}
where $$V(x,y,t)=(x^2-1)^2+y^2+y^4+2x^2\sin\left(\frac{2\pi t}{T_1}\right)+x\sin\left(\frac{2\pi t}{T_2}\right),$$ $T_1=10$, $T_2=7$, $\sigma=1.5$, whilst $W_1$ and $W_2$ are independent Wiener processes. The drift component  explicitly depends on time and has a periodicity of $T=T_1\times T_2= 70$. As a result of the strong time modulation, the qualitative properties of the potential change wildly with time, as illustrated in Fig. \ref{fig:potentialpanel}. The $V(x,y,t)$ is shown at each time unit in a movie included in the Supplementary Information. 

The evolution of the probability density with respect to Lebesgue obeys the following time-dependent Fokker-Planck equation \cite{pavliotis2014book}
\begin{equation} 
\partial_t \rho(x,y,t)=\partial_i(\partial_i V(x,y,t) \rho(x,y,t))+\frac{\sigma^2}{2} \partial_i\partial_i  \rho(x,y,t)=\mathcal{L}_t \rho(x,y,t) \,\label{fp}
\end{equation}
where {\color{black}we take Einstein's convention on repeated indices, and where} $\mathcal{L}_t$ is the time-dependent Fokker-Planck operator. $P^{t_1,t_2}=\mathcal{T}\exp\int_{t_1}^{t_2}\mathrm{d}\tau \mathcal{L}_\tau$ is the Perron-Frobenius operator acting on densities, where $\mathcal{T}$ indicates time-ordering. 

In what follows, we do not consider the continuous time process given by Eqs. \ref{ito}-\ref{fp}, but we rather  consider the unit‑time skeleton of the process. The associated discrete‑time process with unit time step is defined as $(x_n,y_n)$. The discrete‑time Markov process is generated by the discrete-time Perron-Frobenius operator of the form $P^{p,q}$ where $p,q\in\mathbb{Z}$ $p\geq q$. Given the periodicity of the drift term, we have that $P^{p,q}=P^{p+nT,q+nT}$ $\forall$ $n\in\mathbb{Z}$.

Our data analysis proceeds as follows. We first generate, after discarding a transient, a long integration of the system given in Eq. \ref{ito} lasting $N_T\times T$ time units, where $N_T=5\times 10^6$. We adopt the Euler-Maruyama numerical scheme with time step $dt=0.01$, thus generating a bivariate time series with  $N_T T /dt$ time steps. We reduce the database to a time series of $N_T T$ pairs $(x_n,y_n)$ describing the state of the system at integer time units. 

\subsection{Estimating the Equivariant Measure}
We choose as phase space the rectangle $D=[-2.7,2.7]\times[-2.2,2.2]\subset\mathbb{R}^2$. All the datapoints are included in this set. Note that given the nature of the noise, positive - yet entirely negligible - mass is always present outside the considered domain, but this is deemed entirely irrelevant in the current analysis. Alternatively, we could compactify the space by setting the dynamics in the torus corresponding to $D$.

Let \(s_j>0\) denote a prescribed mesh size. In our analysis we choose $s_1=0.3$, $s_2=0.2$, $s_3=0.15$, $s_4=0.1$, $s_5=0.075$, $s_6=0.05$. 
We construct a regular Ulam partition of \(D\) with uniform side length
\(s_j\) in both coordinate directions as follows. Define uniform grids in the \(x\)- and \(y\)-directions by
\[
x_i = -2.7 + i {\color{black}s_j}, \qquad i = 0,1,\dots,N_x,
\]
\[
y_k = -2.2 + k s_j, \qquad k = 0,1,\dots,N_y,
\]
where
\[
N_x = \left\lceil \frac{5.4}{s_j} \right\rceil,
\qquad
N_y = \left\lceil \frac{4.4}{s_j} \right\rceil.
\]
The corresponding regular Ulam partition is given by
\[
\mathcal{U}^j
=
\left\{
B^j_{i+1,k+1}
=
[x_i,x_{i+1}) \times [y_k,y_{k+1})
\;:\;
0 \le i \le` N_x-1,\;
0 \le k \le N_y-1
\right\},
\]
with the convention that sets on the upper boundary of \(D\) are closed on their right and top edges. We now choose a simpler single-index notation $B^j_{l,m}\rightarrow B^j_p$. The collection \(\mathcal{U}^j\) forms a finite measurable partition of \(D\),
that is,
\[
D = \bigcup_{p} B^j_p,
\qquad
B_{p} \cap B_{q} = \emptyset
\quad \text{for } p \neq q.
\]
Each interior element \(B_{p}\) is a square of side length \(s_j\) and Lebesgue measure \(s_j^2\).  We construct a time-dependent Markov model along the lines described in \cite{Lucarini2016,Lucarini2025,Lucarini2026}. Since our system is periodic with period T,  we need to define a different Markov model $\mathcal{M}^j_n$ for each of the $n=1,\ldots,70$ time units and for each Ulam partition $\mathcal{U}^j$. The model at time unit $70$ is identified with the model at time unit 0. 

\begin{figure}
    \centering\includegraphics[width=1.\linewidth]{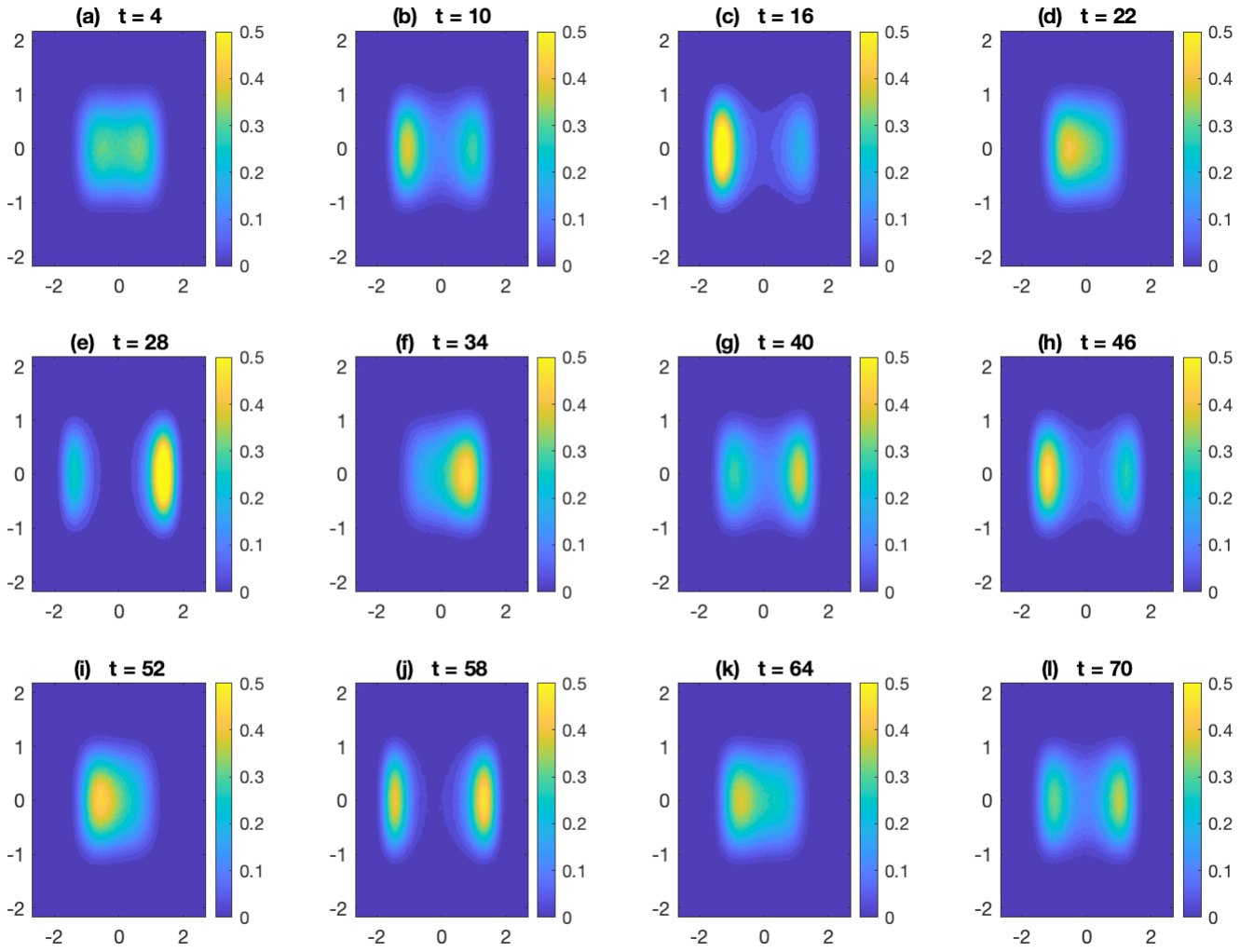}
    \caption{Illustration of the empirical equivariant density $\tilde{h}^6_t$  estimated using with the finest grid with spacing $s_6=0.05$. Full movie available online in the Supplementary Information.}
    \label{fig:equivariantpanel}
\end{figure}

Using standard techniques, we first construct  the Markov models using a frequentist approach by virtue of the classical maximum likelihood estimator.  Let $n^j_{n;p,q}$ be number of observed transitions from state $p$ to state $q$ belonging to the Ulam partition $\mathcal{U}^j$ at time $n$ collected across the  $N_T=5\times10^6$ simulated periods:
\[
\eta^j_{n;p,q} = \sum_{l=0}^{N_{T}-1} \mathbf{1}_{B^j_p}(x_{n+lT}) \text{ }  \mathbf{1}_{B^j_q}(x_{n+1+lT})
\]
and let $\gamma^j_{n;p}=\sum_{l=0}^{N_{T}-1} \mathbf{1}_{B^j_p}(x_{n+lT})$ be the total occupancy of state $p$ at time n.  The maximum likelihood estimate for the discrete Markov model is
\begin{equation}
    \mathcal{M}^j_{n;p,q}= \frac{\eta^j_{n;p,q}}{\gamma^j_{n;p}}.
\end{equation}

\begin{figure}
    \centering
    \includegraphics[width=1.\linewidth]{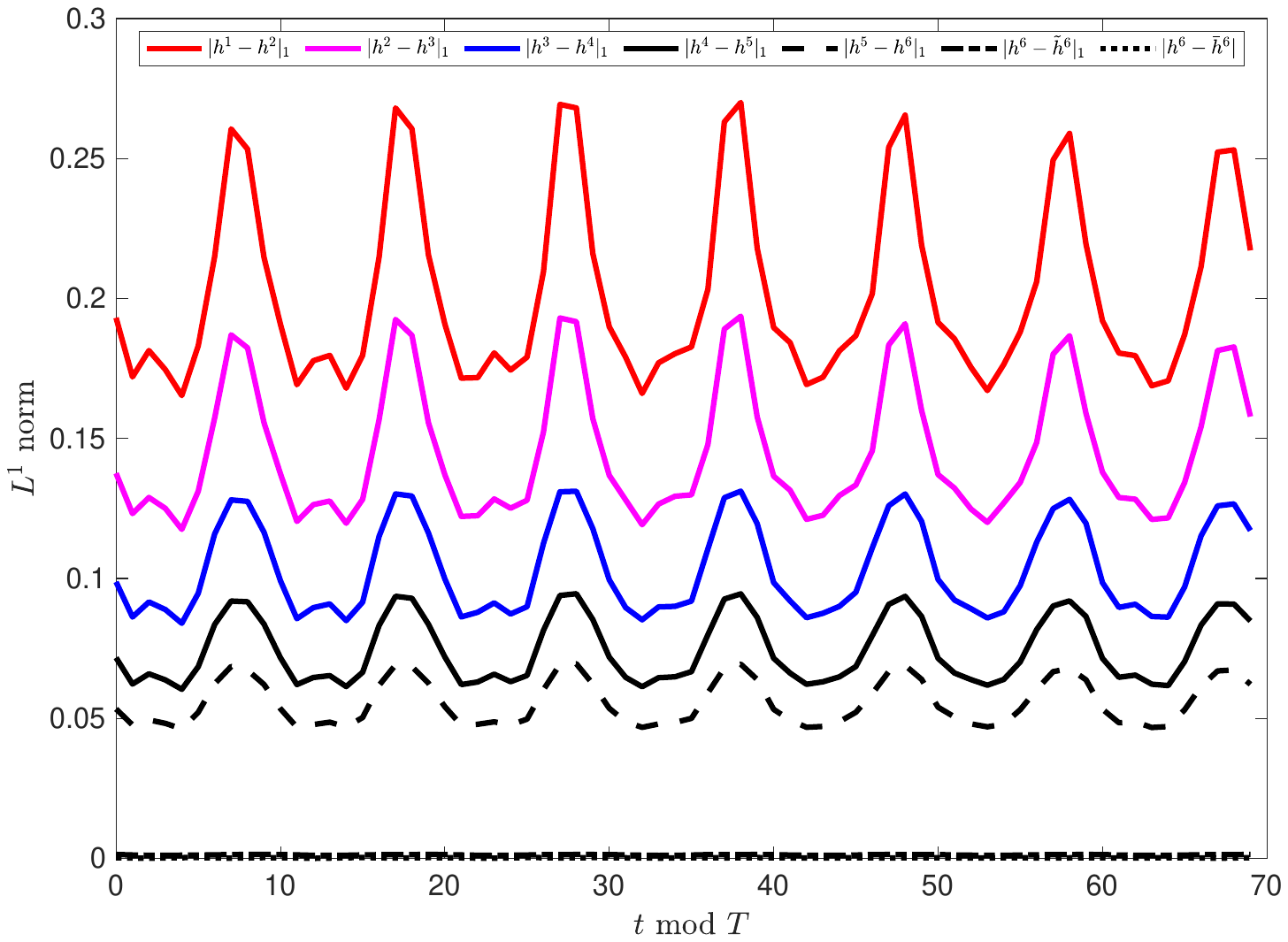}
    \caption{Convergence of the equivariant measure as one consider finer and finer Ulam griddings. We denote with $h^j(t)$ the estimate of the equivariant measure at time $t$ obtained with the gridding $s_j$. $s_1$ corresponds to cells of size $dx_1=0.3$, $s_2$ corresponds to cells of size $dx_2=0.2$, $s_3$ corresponds to cells of size $dx_3=0.15$, $s_4$ corresponds to cells of size $dx_4=0.1$, $s_5$ corresponds to cells of size $dx_5=0.07$, and $s_6$ corresponds to cells of size $dx_6=0.05$. $\tilde{h}_6(t)$ is the frequentist estimate obtained by box counting and corresponds to our ground truth. Finally $\bar{h}_6(t)$ is the estimate of the equivariant measure derived using the Markov model which has been estimated with the Bayesian correction, see text.}
    \label{fig:equivariant}
\end{figure}

\begin{figure}
    \centering
    \includegraphics[width=1.\linewidth]{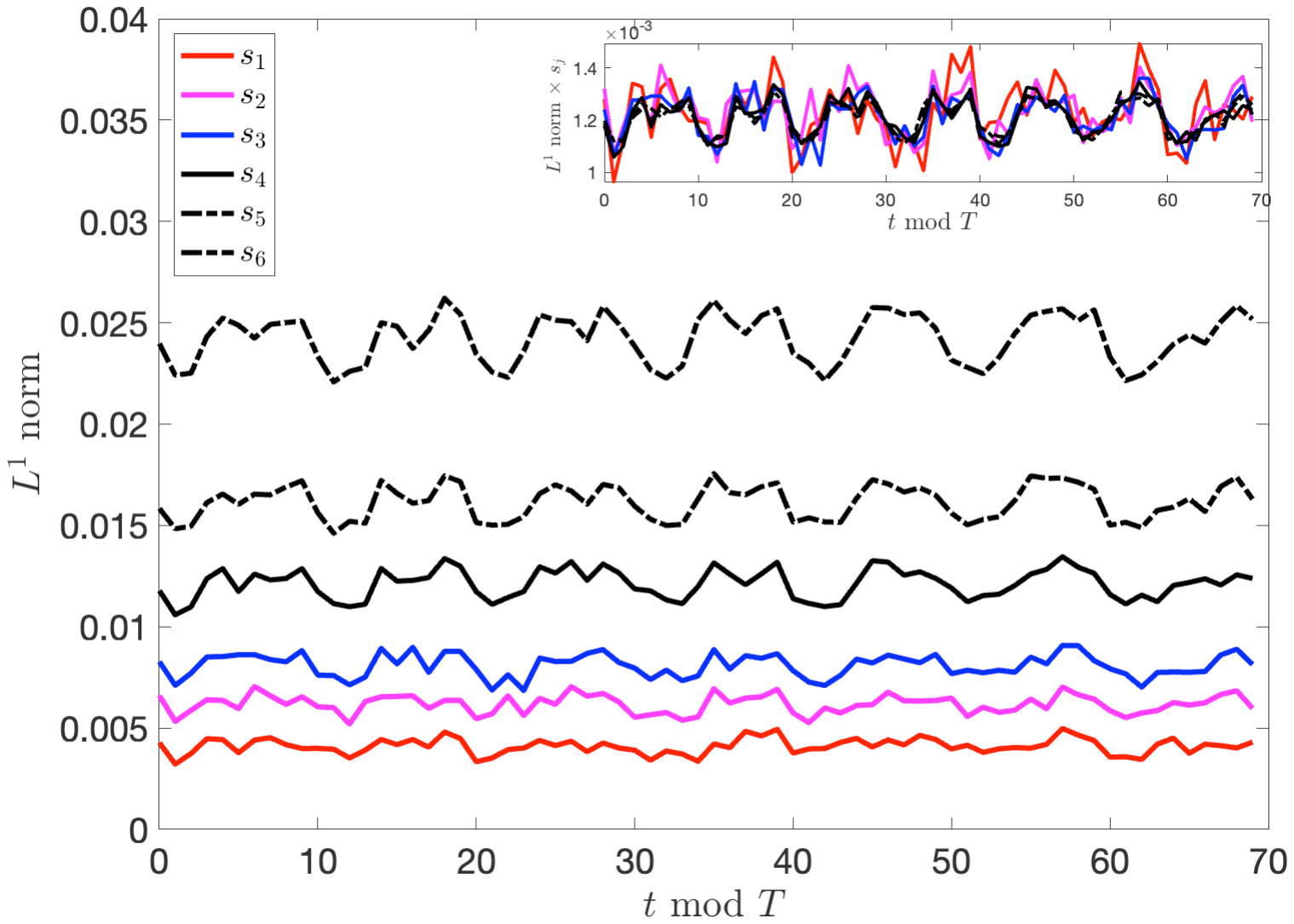}
    {\color{black}\caption{$L^1$ norm of the difference between the equivariant measures $h^j(t)$ obtained with the griddings $s_j$, $j=1,\ldots,6$ obtained using two different realizations of the random system, by two separate simulations, both with total duration $5\times10^6$ $T$. The agreement is better when one considers coarser gridding. When one rescales the norm by multiplying times $s_j$, thus accounting for the uncertainty associated with the sampling error, all curves collapse into one, see inset.}
    \label{fig:ensemble}}
\end{figure}

The equivariant measure $h^j_t$ for at time $t$ mod $T$ for the Ulam partition $\mathcal{U}^j$ is then constructed as follows \cite{Lucarini2026}:
\begin{equation} h^j_n = \lim_{k\rightarrow\infty} \mathcal{M}^j_{n-1}\mathcal{M}^j_{n-2}\ldots \mathcal{M}^j_{n-k} \mu^j, \label{covariantMarkov}
\end{equation}
where $\mu^j$ is in principle any probability measure on $\mathcal{U}^j$. Practically, $k$ in the equation above is set to $2\times T$ and $\mu^j$ is the eigenvector of unitary eigenvalue of $\mathcal{M}^j_T$. Since our specific system has a rapid memory loss, the obtained equivariant measure is robust as long as $m \gtrsim T$.

Note that instead $\tilde{h}^j_t(p)=\gamma^j_{t;p}/N_T$  is the frequentist estimate of the equivariant density for the state $p$ at time $t$. An illustration of $\tilde{h}^6_t$ is shown in Fig. \ref{fig:equivariantpanel}, 
while a full movie is available online in the Supplementary Information.

Figure \ref{fig:equivariant} shows that the $L^1$ norm of the difference between the equivariant measures obtained for progressively finer and finer partition decreases as the gridding becomes finer. We also remark that the agreement is worse - especially in the case of coarser partitions - with periodicity $T_1=10$, see the periodic spikes in the $L^1$ norm, because the corresponding time modulations lead to creation of finer structures in the equivariant measure, which are harder to capture at low resolution. Reassuringly, this effect is greatly reduced as we consider very fine partitions. Finally, we observe that when considering the finest gridding, the frequentist estimate of the equivariant measure (which we take as our ground truth) is in almost exact agreement with the equivariant measure constructed by iterating the time-dependent Markov model. This provides a strong numerical support of the convergence properties discussed above. The equivariant measures  $h^j_t$ for $j=1,\ldots,6$ are reported as movies in the Supplementary Information.

{\color{black}Whilst the numerical investigation considered here is just a proof of concept  and a practical counterpart of the analytical work dicussed in the rest of the paper, it seems relevant to provide some information on the associated uncertainties. At this purpose, we have performed a second simulation of the system given in Eq. \ref{ito} having a duration $N_T\times T$. Figure \ref{fig:ensemble} shows the $L^1$ difference between the equivariant measures $h^j_t$ estimated from the two datasets using the same protocol described above. Resulting discrepancies result exclusively from the fact that we are considering a large but finite number of data points. The agreement depends on the phase $t$ mod $T$ of the system and qualitatively follows the same dominant periodicity $T_1=10$ discussed in Fig. \ref{fig:equivariant}. The agreement is in general rather good - compare the vertical scale with the one used in Fig. \ref{fig:equivariant} - and it improves as we consider larger values of $s_j$, as a result of coarse graining. Interestingly, the curves obtained for different values of $s_j$ collapse onto each other as we factor in the fact that the average occupation of the grid cells scales proportionally with $s_j^2$, so that the relative sampling error scales with the inverse of $s_j$, see inset. This is a manifestation of a  \textit{no-free-lunch} principle: if we want to consider higher resolutions gridding, we need longer time series to achieve the same accuracy on the resulting equivariant measure. }

{\color{black}From now on we do not consider anymore the second set of simulations of the reference system.} Note that in order to increase the robustness of the estimated Markov model, one can perform a Bayesian correction for the estimate of the Markov chain $\mathcal{M}_j$. Given the  strong acting noise, our simulations can be assumed to consist of $N_T$ simulations each with duration $T$.  We follow \cite{Diaconis2006,Lucarini2026} and introduce a Dirichlet factor associated with creating $\sqrt{N_{T}}$ pseudo-observations distributed according to the estimated equivariant measure $h^j_t$. Hence, we obtain a revised estimate of the stochastic matrix at time t,
\begin{equation}
\mathcal{M}^j_{t;l,k}= \frac{n^j_{t;l,k}+\sqrt{N_{T}}h^j_k}{n^j_{t;l}+\sqrt{N_{T}}}.
\end{equation}
The updated estimate of $\mathcal{M}^j_{t}$ can then be used in Eq. \ref{covariantMarkov} to compute a new approximation of the equivariant measure, which we indicate as $\bar{h}^j_t$.  so that the procedure can be iterated. This in our case, no noticeable change in the equivariant measure is obtained after the first iteration, see Fig. \ref{fig:equivariant}, so that we can safely assume  $\bar{h}^j_t={h}^j_t$. The Bayesian correction, to the estimate of the Markov model, {which makes sure that the system obeys Assumption \ref{ass:positive_smoothing_kernels},}  will prove essential in the next step of our investigation.

We remark that in our Markov analysis we consider a number of states ranging from $\mathcal{O}(200)$ for the $h_1$ gridding up to $\mathcal{O}(5000)$ for the $h_6$ gridding. Hence, in terms of both spatial resolution and temporal length the current analysis is order of magnitude more extensive that the previous work where we had constructed time dependent Markov models and their response to perturbations \cite{Lucarini2026}.

\subsection{Estimating the Response to Perturbations}

We now investigate the response of the system to perturbations. We consider a rather simple scenario where the forcing corresponds to a small change in the potential $V$ of the form $V(x,y,t)\rightarrow V(x,y,t)+\epsilon W(x,y,t)$ where $W(x,y,t)=y$. This amounts to adding a small constant term of intensity $-\epsilon$ in the drift of the $y$ variable. This perturbation breaks the symmetry with parity -1 of the drift in the $y$  direction and is expected to create a displacement of probability towards the negative values of $y$. Since the extra forcing is time-independent, also the perturbed dynamics has periodicity $T$. We follow for the perturbed system the same protocol described in the previous subsection for the unperturbed case and generate a  time series of $N_T T$ pairs $(x^\epsilon_n,y^\epsilon_n)$ describing the state of the system at integer time units. We choose $\epsilon=0.2$.

Following the same procedure used in the previous subsection allows us to estimate, for each choice of $s_j$, $j=1\ldots,6$, the Markov model $\mathcal{M}^{j,\epsilon}_{t;p,q}$. The Markov models $\mathcal{M}^{j,\epsilon}_{t;p,q}$, $t=1,\ldots,70$ can then be used to obtain the equivariant measure $h^{j,\epsilon}_t$ using Eq. \ref{covariantMarkov}.  

We can then derive the following estimate for the first order correction to the stochastic matrix: 
\begin{equation}
{m}^{j}_{t;p,q}\approx\frac{
 \mathcal{M}^{j,\epsilon}_{t;p,q}-\mathcal{M}^{j}_{t;p,q}}{\epsilon}
\end{equation}
where ${m}^{j}_{t}$ is the linear correction to the stochastic matrix associated with the perturbation. The use of the Bayesian correction described above for the estimate of the Markov chain $\mathcal{M}_j$ is important to avoid pathologies when estimating $m^j_t$. Indeed, the Bayesian correction ensures that all columns of $\mathcal{M}^{j,\epsilon}$ and  $\mathcal{M}^{j}$ sum up to one, so that the columns of $m^j_n$ sum up to zero. The Bayesian correction regularises the computation of the Markov kernels and, since in our case $N_T=5
\times10^6$, has a negligible effect on the statistics of the system, as mentioned above.

As discussed in \cite{Lucarini2026}, by using recursion it is possible to find a simple explicit expression for $\hat h_t $ as follows:
\begin{equation}
\hat h^j_t  = \sum_{k=-\infty}^{\infty}\Theta(k) \mathcal{M}^j_{t-1}\ldots\mathcal{M}^j_{t-k} m^j_{t-k-1} \tilde h^j_{t-k-1} \label{formulaMarkovmeasures}
\end{equation}
where $\Theta(k)=1$ if $k\geq0$ and $\Theta(k)=0$ otherwise. Note that the presence of $\Theta(k)$ ensures that the system obeys causality. The summation above is truncated to $k_{max}=2\times T$. Our estimate for the linear response obtained using Eq. \ref{formulaMarkovmeasures} for the finest gridding $s_j=0.05$ is reported in Fig. \ref{fig:equivariantpanelbis}. As we see, the effect of the perturbation is to move mass from the positive to the negative y-domain. The linear response obtained using  
Eq. \ref{formulaMarkovmeasures} compares well with the finite  difference $\Delta h^{j}_t=(h^{j,\epsilon}_t-h^{j}_t)/\epsilon$, as illustrated in Fig. \ref{fig:equivariantpanel2} for the higher resolution case $j=6$. We have that $\max_t\left(|\Delta h^{6}_t-\hat{h}^{6}_t|_{L^1}/\left(|\hat{h}^{6}_t|_L^{1}+|\Delta h^{6}_t|_L^{1}\right)\right)\approx0.025$.

\begin{figure}
    \centering\includegraphics[width=1\linewidth]{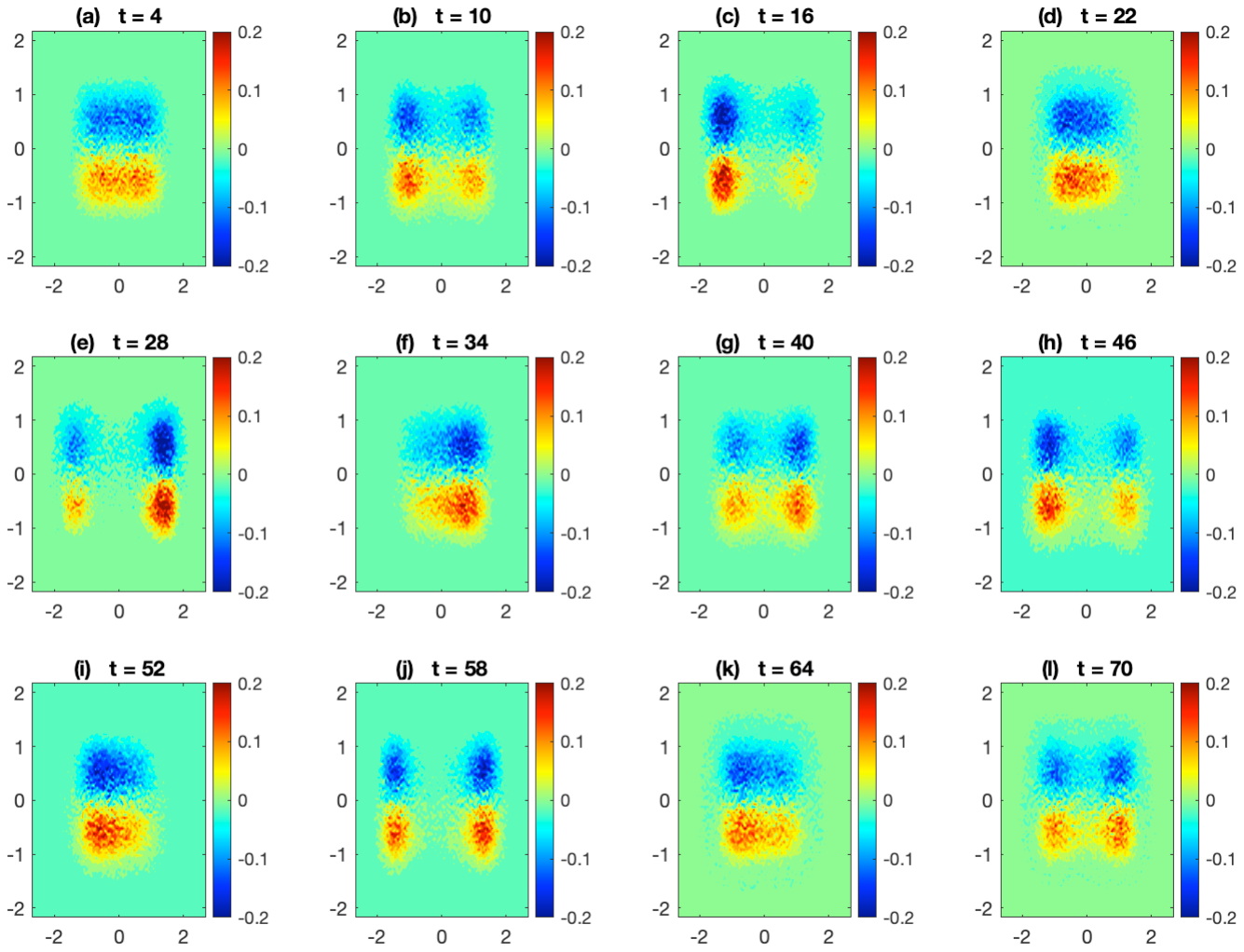}
    \caption{Illustration of the predicted linear response $\hat h^6_t$ obtained using Eq.\ref{formulaMarkovmeasures} for the finest grid with spacing $s_6=0.05$. Full movie available online in the Supplementary Information.}
    \label{fig:equivariantpanelbis}
\end{figure}

\begin{figure}
    \centering\includegraphics[width=1\linewidth]{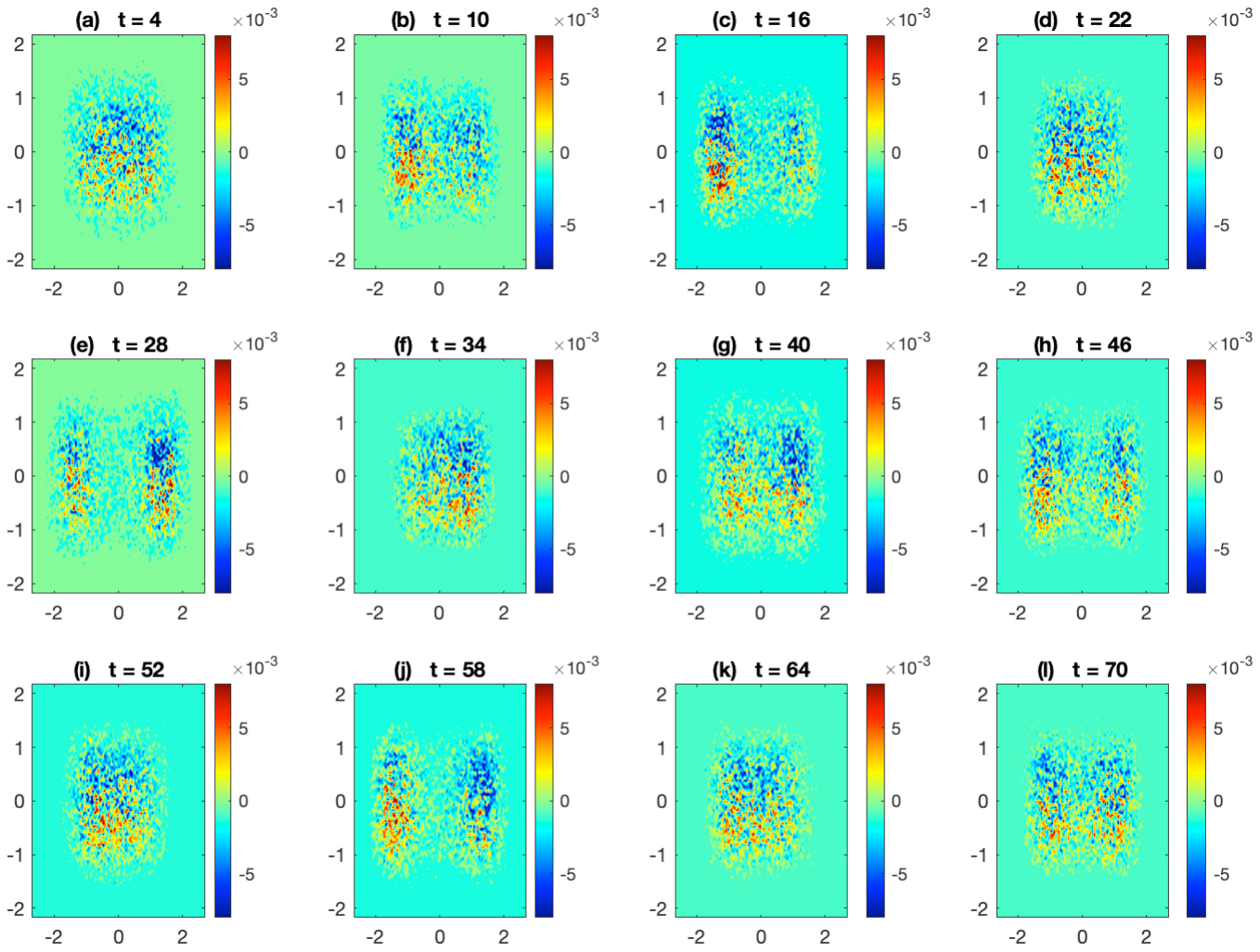}
    \caption{Illustration of $\Delta h^6_t-\hat h^6_t$. We consider here the finest grid with spacing $s_6=0.05$. Note the much smaller range of values for the field as compared to Fig. \ref{fig:equivariantpanel}, which shows that the linear response performs rather well. Full movie available online in the Supplementary Information.}
    \label{fig:equivariantpanel2}
\end{figure}

\section{Conclusions}\label{sec:conclu}

Nonautonomous systems arise naturally in many applications, but their statistical analysis and coarse-grained approximation remain much less developed than in the autonomous case \cite{Ashwin2026,Crisan2026}. In previous investigations we have defined under which conditions it is possible to develop a response theory that describes how their statistical properties change as a result of an additional weak forcing \cite{Lucarini2026,GL26}, in this work we addressed the theoretically and practically relevant question of the meaningfulness of studying such systems using a coarse-graining strategy described by  Ulam-like partitions. The matter is central to many applications because Markov modeling and Ulam-like partitions are explicitly or implicitly used to study a myriad of complex systems for which ultra higher resolution numerical modeling is unfeasible or in many data-driven settings.

Fortunately, our findings support a couple of positive and encouraging messages:
\begin{itemize}
\item coarse-graining procedures associated with the Ulam method and its generalizations discussed above allow one to capture accurately the equivariant measure in nonautonomous systems; and 
\item for systems whose evolution operators are regularizing, finite element reduction allows for capturing also the linear response operators for a rather general class of perturbations. To the best of our knowledge, a general approximation result of this type, formulated directly in terms of finite-state Markov reductions, had not previously been established in this form, even in the autonomous setting. 
\end{itemize}
The theoretical results are complemented by numerical experiments on a simple but instructive time-dependent diffusive model, illustrating the convergence of the reduced Markov description for both the equivariant family and its response to perturbations. {\color{black}In future works we aim at performing accurate investigations of more complex numerical models using the framework discussed in this paper and at devoting special attention to studying the uncertainties resulting from the finiteness of the available data or from the presence of uncertainties in their value.}

We hope that our work{, which bears direct relevance also for study of Koopman/Kolmogorov operator theory-based  approximate representations of dynamical systems \cite{Budisic2012,Mauroy2020} and of their response to perturbations \cite{Santos2022,lucarini2025generalframeworklinkingfree,zagli_SIAM:2026}}, will contribute to creating better tools for investigating fluctuations, response, and critical behaviour in complex systems.
Natural directions for future work include weakening the regularity assumptions used here, extending the approximation results to broader classes of nonautonomous systems, and developing a parallel theory for time-dependent or data-adapted projection schemes of the kind suggested in Remark \ref{last_rmk}.

\section*{Data Availability}
The data produced in the numerical simulations discussed in Sect. \ref{numericalsim} and the source files of all the figures included in the paper can be accessed as Supplementary Information that can be found here: \texttt{https://doi.org/10.6084/m9.figshare.32090992}.

\section*{Acknowledgements}
VL acknowledges partial support received from the ARIA SCOP-PR01-P003-Advancing Tipping Point Early Warning AdvanTip project, from the Horizon Europe Projects Past2Future (Grant No. 101184070) and ClimTIP (Grant No. 100018693), from the European Space Agency Project PREDICT (Contract 4000146344/24/I-LR), and from the NSFC  International Collaboration Fund for Creative Research Teams (Grant No. W2541005).
SG acknowledges the MIUR Excellence Department Project awarded to the
Department of Mathematics, University of Pisa, CUP I57G22000700001.

{\color{black}\appendix
\section{An Example of Explicit Computation of the Equivariant Measure}\label{AppGauss}
Consider the non-autonomous stochastic differential equation (SDE):
\begin{equation}
dx(t) = (-x(t) + h(t))dt + \sigma dW_t
\end{equation}
where $x\in\mathbb{R}$, $h(t)$ is a time-dependent deterministic forcing term, $W_t$ is a standard Brownian motion, and $\sigma > 0$ represents the noise intensity. This is an elliptic diffusion process, whose  corresponding Fokker-Planck equation  governing the evolution of the probability density function $p(x,t)$ is given by \cite{pavliotis2014book}:
\begin{equation}
\frac{\partial p}{\partial t} = -\frac{\partial}{\partial x} \Big[ (-x + h(t))p \Big] + \frac{\sigma^2}{2} \frac{\partial^2 p}{\partial x^2}.
\end{equation}
Let $\mu_s(x_0)$ be an arbitrary initial probability measure at time $s$. By the linearity of the Fokker-Planck equation, the evolved distribution at the observation time $t\geq s$ is obtained via convolution:
\begin{equation}
p(x, t) = \int_{\mathbb{R}} p(x, t \mid x_0, s) \, d\mu_s(x_0)
\end{equation}
where the transition density $p(x, t \mid x_0, s)$ remains Gaussian for all $t \ge s$:
\begin{equation}
p(x, t \mid x_0, s) = \frac{1}{\sqrt{2\pi v(t; s)}} \exp\left( -\frac{(x - m(t; s, x_0))^2}{2v(t; s)} \right)
\end{equation}
where the explicit time-forward mean $m(t; s, x_0)$ and variance $v(t; s)$ are:
\begin{align}
m(t; s, x_0) &= x_0 e^{-(t-s)} + \int_s^t e^{-(t-\tau)} h(\tau) \, d\tau \\
v(t; s) &= \frac{\sigma^2}{2} \left( 1 - e^{-2(t-s)} \right)
\end{align}
To compute the equivariant measure, we fix the present observation time $t$ and pull the initial starting point into the distant past by taking the limit $s \to -\infty$. One obtains:
\begin{align}
\lim_{s \to -\infty} x_0 e^{-(t-s)} &= 0 \quad \forall x_0 \in \mathbb{R} \\
\lim_{s \to -\infty} m(t; s, x_0) &= A(t) \equiv \int_{-\infty}^t e^{-(t-\tau)} h(\tau) \, d\tau \\
\lim_{s \to -\infty} v(t; s) &= \frac{\sigma^2}{2}
\end{align}
As a result, the transition kernel completely loses its dependence on the initial configuration $x_0$:
\begin{equation}
\lim_{s \to -\infty} p(x, t \mid x_0, s) = p_{\text{equi}}(x, t) = \frac{1}{\sqrt{\pi \sigma^2}} \exp\left( -\frac{(x - A(t))^2}{\sigma^2} \right)
\end{equation}

Substituting this pointwise limit back into the total density integral allows us to extract the pullback profile from the integration operator:
\begin{align}
\lim_{s \to -\infty} p(x, t) &= \int_{\mathbb{R}} \lim_{s \to -\infty} p(x, t \mid x_0, s) \, d\mu_s(x_0) \\
&= p_{\text{equi}}(x, t) \int_{\mathbb{R}} d\mu_s(x_0)=p_{\text{equi}}(x, t),
\end{align}
which proves that the equivariant measure at time $t$ is well defined. 
}

 \section{The classical Ulam discretization in the \(BV\)–\(L^1\) setting}
\label{subsec:ulam_bv_l1}

\begin{figure}
    \centering
    \includegraphics[trim={5cm 9.0cm 5cm 6.9cm}, clip,width=0.45\linewidth]{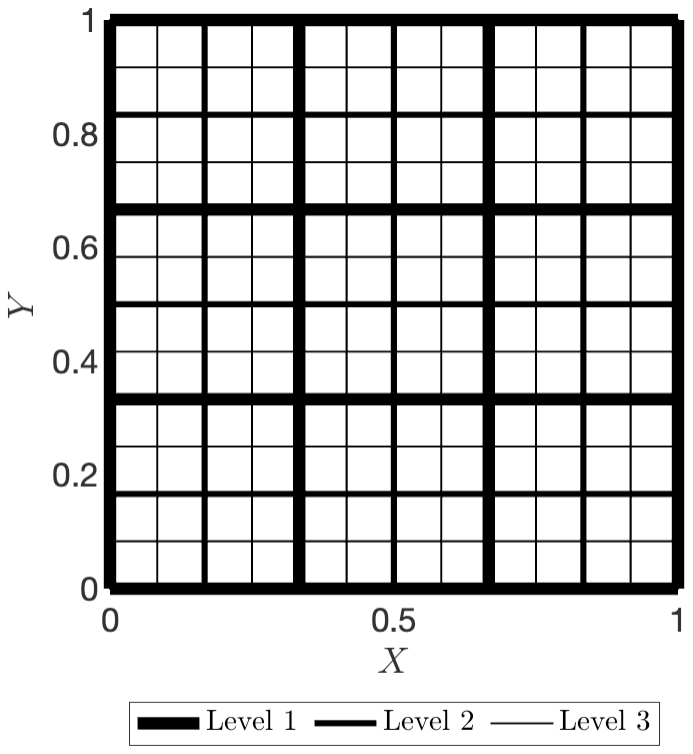}
\includegraphics[trim={5cm 7.5cm 4cm 9cm},clip,width=0.45\linewidth]{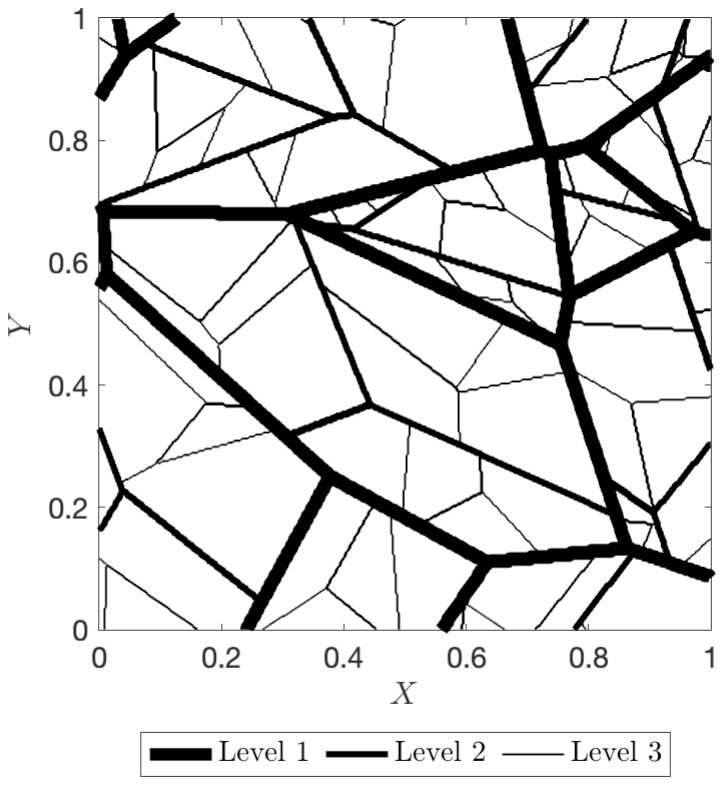}
    \caption{Two {\color{black}possible variants of Ulam} partitions, where sequential refinements are indicated by the levels. Left panel: regular partition. Right panel: irregular partitions, possibly Voronoi tessellation. The units square in $\mathbb{R}^2$ is used as domain for illustrative purposes. {\color{black}Each partition defines a different projection on piece-wise constant densities. However, all projections of this kind satisfy the same properties as requested in Appendix \ref{subsec:ulam_bv_l1}.}}
    \label{fig:Ulam}
\end{figure}

In this section, {\color{black} we recall some well known estimates, to} discuss the classical Ulam discretization as a concrete example of a finite-rank projection scheme satisfying the abstract assumptions {\color{black}${\bf (P1),(P2),(M)}$} used in the main text. 

The purpose is twofold: on the one hand, it provides a meaningful and well-known class of discretizations to which the general theory applies; on the other hand, it clarifies the role of the abstract assumptions we made.
  
We work on the unit interval
\[
X=[0,1],
\]
endowed with Lebesgue measure \(m\). We identify absolutely continuous signed measures with their densities, and consider the weak and strong spaces
\[
B_w:=L^1([0,1]),
\qquad
\|f\|_w:=\|f\|_{L^1},
\]
and
\[
B_s:=BV([0,1]),
\qquad
\|f\|_s:=\|f\|_{BV}:=\|f\|_{L^1}+Var(f).
\]
The corresponding zero-mass spaces are
\[
V_w:=\left\{f\in L^1([0,1]):\int_0^1 f\,dm=0\right\},
\qquad
V_s:=BV([0,1])\cap V_w.
\]

Let
\[
\mathcal I_m=\{I_1,\dots,I_m\}
\]
be the uniform partition of \([0,1]\) into intervals of length
\[
|I_j|=\delta_m:=\frac1m.
\]
Let \(\mathcal F_m\) be the \(\sigma\)-algebra generated by \(\mathcal I_m\), and let
\[
\pi_m:L^1([0,1])\to L^1([0,1])
\]
be the conditional expectation with respect to \(\mathcal F_m\),
\[
\pi_m(f)=\mathbb E(f\mid\mathcal F_m).
\]
Equivalently, \(\pi_m(f)\) is the piecewise constant function which on each \(I_j\) is equal to the average of \(f\) on \(I_j\). Given a transfer operator \(L\), the associated classical Ulam discretization is
\[
M_m:=\pi_mL\pi_m.
\]

A graphical illustration of possible practical strategies for implementing the Ulam method at different levels of (sequential) refinements is given in Fig. \ref{fig:Ulam}, where we show both the case of regular gridding and of Voronoi tessellations \cite{Aurenhammer91} in the unit square in $\mathbb{R}^2$. {\color{black}Despite the geometric differences, all projections of this sort obey the properties required for obtaining convergence in the limit of infinitely fine partitions. A detailed comparison of how choosing either of these two  Ulam gridding strategies affects the quality of the response operators and a detailed analysis of the added value of considering finer and finer partitions is presented in  \cite{lucarini2025generalframeworklinkingfree}.}

The choice \(B_s=BV\) is natural in this context: while \(\pi_m\) does not preserve smoother spaces such as \(W^{1,1}\), it does preserve \(BV\), and the corresponding approximation estimates are quantitative.

\begin{lemma}[Basic properties of the Ulam projection]
\label{lem:ulam_projection_bv_appendix}
For every \(m\ge1\), the operator \(\pi_m\) satisfies:
\begin{enumerate}
\item
\[
\|\pi_m f\|_{L^1}\le \|f\|_{L^1}
\qquad\forall f\in L^1([0,1]);
\]
\item
\[
\|\pi_m f\|_{BV}\le \|f\|_{BV}
\qquad\forall f\in BV([0,1]);
\]
\item
\[
\|\pi_m f-f\|_{L^1}\le \delta_m\,Var(f)
\le \delta_m\,\|f\|_{BV}
\qquad\forall f\in BV([0,1]).
\]
\end{enumerate}
\end{lemma}

\begin{proof}
The \(L^1\)-contraction follows from the contractivity of conditional expectation.

For the \(BV\) estimate, note that \(\pi_m f\) is obtained by averaging \(f\) on the partition elements. The jump of \(\pi_m f\) across adjacent intervals is therefore bounded by the oscillation of \(f\) on the corresponding region, and hence
\[
Var(\pi_m f)\le Var(f).
\]
Together with the \(L^1\)-contraction, this gives
\[
\|\pi_m f\|_{BV}\le \|f\|_{BV}.
\]

Finally, on each interval \(I_j\), the value of \(\pi_m f\) lies between the essential infimum and essential supremum of \(f\) on \(I_j\). Thus
\[
\int_{I_j}|\pi_m f-f|\,dm
\le
|I_j|\bigl(\sup_{I_j}f-\inf_{I_j}f\bigr).
\]
Summing over \(j\), we obtain
\[
\|\pi_m f-f\|_{L^1}
\le
\delta_m\sum_j\bigl(\sup_{I_j}f-\inf_{I_j}f\bigr)
\le
\delta_m\,Var(f),
\]
as claimed.
\end{proof}

The previous lemma shows that the classical Ulam projections satisfy the abstract projection assumptions used in the main text.

\begin{remark}
The classical Ulam scheme is the canonical example of the abstract projection framework considered in this paper. The same strong-weak mechanism, however, also applies to smoother finite-rank approximations once suitable analogues of the estimates in Lemma~\ref{lem:ulam_projection_bv_appendix} are available.
\end{remark}

\begin{remark}[A time-dependent empirical Ulam scheme]\label{last_rmk}
A realistic model of the empirical Ulam method applied in Section \ref{numericalsim} to derive a Markov transition matrix from long orbits in the non-autonomous context can be obtained by considering time-dependent finite-state Markov models adapted to the statistical state of the system at each time, where the transition probabilities between partition elements are described through averages with respect to the corresponding equivariant measure.
 In an idealized deterministic formulation, this corresponds to replacing the Lebesgue-based projection by a family of projections defined through conditional expectations with respect to the equivariant family. More precisely, let \((\nu_n)_{n\in\mathbb Z}\) be a reference equivariant family, let
\[
\mathcal I_m=\{I_1,\dots,I_{N_m}\}
\]
be a finite partition of the phase space, and let \(\mathcal F_m\) be the \(\sigma\)-algebra generated by this partition. For each \(n\in\mathbb Z\), one may then define a time-dependent projection by conditional expectation with respect to \(\nu_n\),
\[
\pi_{n,m}f:=\mathbb E_{\nu_n}(f\mid \mathcal F_m).
\]
Equivalently, whenever \(\nu_n(I_j)>0\) for all \(j\),
\[
(\pi_{n,m}f)(x)
=
\sum_{j=1}^{N_m}
\left(
\frac{1}{\nu_n(I_j)}\int_{I_j} f\,d\nu_n
\right)\mathbf 1_{I_j}(x).
\]
The associated reduced dynamics would then be given by the time-dependent discretization
\[
M_n^{(m)}:=\pi_{n+1,m}\,L_n\,\pi_{n,m}.
\]

A theory parallel to the one developed in the present paper appears possible for such time-dependent projection schemes. We leave this extension for future work, so as not to further broaden the scope of the paper.
\end{remark}

\def\cprime{$'$} \def\cprime{$'$} \def\cprime{$'$} \def\cprime{$'$}
  \def\cprime{$'$} \def\cprime{$'$} \def\cprime{$'$} \def\cprime{$'$}
  \def\Rom#1{\uppercase\expandafter{\romannumeral #1}}\def\u#1{{\accent"15
  #1}}\def\Rom#1{\uppercase\expandafter{\romannumeral #1}}\def\u#1{{\accent"15
  #1}}\def\cprime{$'$} \def\cprime{$'$} \def\cprime{$'$} \def\cprime{$'$}
  \def\cprime{$'$} \def\cprime{$'$} \def\cprime{$'$}
  \def\polhk#1{\setbox0=\hbox{#1}{\ooalign{\hidewidth
  \lower1.5ex\hbox{`}\hidewidth\crcr\unhbox0}}} \def\cprime{$'$}
  \def\cprime{$'$} \def\cprime{$'$}

\end{document}